\documentclass{article} 
\usepackage[margin=2.5cm]{geometry} 
\usepackage[T1]{fontenc} 
\usepackage{anyfontsize} 
\usepackage{ragged2e} 

\usepackage{amsfonts} 
\usepackage{amssymb, amsmath} 
\usepackage{amsthm} 
\usepackage{bbm} 
\numberwithin{equation}{section} 

\usepackage{graphicx} 
\usepackage{caption} 
\usepackage{subfig} 
\usepackage{subcaption} 
\usepackage{placeins} 
\usepackage{epstopdf} 
\usepackage{array,booktabs,multirow} 
\usepackage{booktabs} 
\usepackage{colortbl}

\usepackage{hyperref} 
\hypersetup{colorlinks=true, linkcolor=blue, filecolor=magenta, urlcolor=cyan, citecolor=teal,}
\newcommand{\citenote}[2]{\cite[#1]{#2}} 

\usepackage{xcolor} 
\definecolor{mygreen}{RGB}{28,172,0} 
\definecolor{mylilas}{RGB}{170,55,241} 
\usepackage{listings} 
\usepackage{extpfeil} 
\usepackage{cancel} 

\theoremstyle{definition}  
\newtheorem{theorem}{Theorem}[section]  
\newtheorem{definition}[theorem]{Definition}  
\newtheorem{lemma}[theorem]{Lemma}  

\begin{document}

\title{Exponential stability of finite-$N$ consensus-based optimization} 

\author{S. G\"ottlich\footnotemark[1], \; J. Heieck\footnotemark[1], \; A. Neuenkirch\footnotemark[1]}

\footnotetext[1]{University of Mannheim, Department of Mathematics, 68131 Mannheim, Germany (goettlich@math.uni-mannheim.de, jacob.heieck@uni-mannheim.de, neuenkirch@uni-mannheim.de).}

\date{\today}

\maketitle

\begin{abstract}
	We study the finite-agent behavior of Consensus-Based Optimization (CBO), a recent metaheuristic for the global minimization of a function, that combines drift toward a consensus estimate with stochastic exploration. While previous analyses focus on asymptotic mean-field limits, we investigate the stability properties of CBO for finite population size \( N \). Following a hierarchical approach, we first analyze a deterministic formulation of the algorithm and then extend our results to the fully stochastic setting governed by a system of stochastic differential equations. Our analysis reveals that essential stability properties, including almost sure and mean square exponential convergence, persist in both regimes and provides sharp quantitative estimates on the rates of convergence.
\end{abstract}



\section{Introduction}
\label{sec:introduction}

\subsection{Outline}
\label{sec:outline}

A central problem in applied mathematics is the global minimization of a potentially non-convex and non-smooth function $f: \mathbb{R}^D \to \mathbb{R}$, along with the search for its global minimizer $\mathbf{X}^{\star}\in\mathbb{R}^{D}$. Computing $f$ or $\mathbf{X}^{\star}$ is generally NP-hard, but practical instances often permit approximate solutions within reasonable accuracy and time.

Metaheuristics \cite{metaheuristics} have proven effective in tackling such optimization problems, orchestrating local improvements and global strategies through a mix of random and deterministic decisions. Prominent examples include Random Search \cite{random_search}, Particle Swarm Optimization (PSO) \cite{particle_swarm_optimization}, Wind-Driven Optimization \cite{wind_driven_optimization} and Ant Colony Optimization \cite{ant_colony_optimization}. Despite their empirical success, many metaheuristics lack rigorous mathematical guarantees for global convergence. Nevertheless, for some of them, such as Random Search, there exist probabilistic guarantees for global convergence \cite{random_search_convergence}.

Recently, Consensus-Based Optimization (CBO) \cite{CBO_origin, carrillo_analytical_framework_CBO} has emerged as a promising alternative, following the principles of metaheuristics while offering greater mathematical tractability. CBO models a finite set $X_t^1, \dots, X_t^N \in \mathbb{R}^{D}$ of $N\in\mathbb{N}$ agents as stochastic processes in time $t\geq0$ exploring the domain and forming a consensus. Their dynamics are driven by two competing forces: a drift term, which pulls agents toward a weighted consensus estimate $\nu_{f}^{\alpha}(X_{t})$ of $\mathbf{X}^{\star}$ given by
\begin{equation}\label{eq:consensus_point}
    \nu_{f}^{\alpha}(X_{t}) = \sum_{m=1}^{N} X_{t}^{m} \frac{\exp (-\alpha f(X_{t}^{m}))}{\sum_{k=1}^{N} \exp (-\alpha f(X_{t}^{k}))}
\end{equation}
with $\alpha>0$, contracting their convex hull; and a diffusion term, modeled as scaled Brownian motion in $\mathbb{R}^{D}$, enhancing exploration. Mathematically, the evolution of each agent $X_t^{n}$ is governed by the stochastic differential equation (SDE)
\begin{equation}\label{CBO_system_description} 
    d X_{t}^{n} = -\lambda\left(X_{t}^{n}-\nu_{f}^{\alpha}(X_{t})\right) d t + \sigma \left\|X_{t}^{n}-\nu_{f}^{\alpha}(X_{t})\right\| d W_t^{n}, \quad n,\ldots,N, 
\end{equation}
where $\lambda >0$ controls the strength of the drift towards the weighted consensus point $\nu_{f}^{\alpha}(X_{t})$, and $\sigma >0$ scales the noise term modeled by $D$-dimensional independent Brownian motions $W_t^{n}$. This formulation ensures a balance between exploitation of promising regions and exploration of the search space, ultimately leading agents to a near-optimal global consensus.

CBO has been applied in several different problem settings. Examples include constrained optimization, such as model predictive control \cite{constrained_CBO, CBO_for_model_predictive_control} and optimization on manifolds like the sphere \cite{fornasier_CBO_Sphere}, multi-objective optimization \cite{multi_objective_CBO}, as well as training neural networks \cite{carrillo_CBO_for_high_dimensional_ML_problems} and 3D human pose estimation in camera coordinates \cite{luvizon_CBO_3D_Human_Pose}. These examples show that CBO can be used to address a broad class of non-convex optimization problems.

The theoretical analysis of CBO has seen significant progress, particularly in the context of its mean-field limit. In \cite{CBO_origin}, the authors establish that as the number of agents $N$ tends to infinity and $\alpha$ approaches infinity (Laplacian principle \cite{laplace_principle}), the system converges arbitrarily close to the minimum. Furthermore, in \cite{fornasier_CBO_converge_globally}, rigorous results are provided to demonstrate that under suitable assumptions, the mean-field formulation of CBO converges globally to the minimizer $\mathbf{X}^{\star}$. These results offer deep insights into the large-population behavior of CBO but rely on asymptotic limits. Apart from the well-posedness of the finite-agent system, which has been established in \cite{carrillo_analytical_framework_CBO}, little is known about the theoretical properties of CBO for finite $N$ and finite $\alpha$. Some progress in this direction has been made for simplified variants of CBO where all agents are influenced by the same random fluctuations, effectively removing the individual stochasticity. While these works provide valuable insights into the convergence properties and error bounds \cite{convergence_first_order_CBO, convergence_first_order_time_discrete_CBO} of these simplified algorithms, they differ structurally from the classical CBO dynamics we study here, which model independent agent behavior more faithfully.

In this work, we take a next step towards further bridging this gap by analyzing the stability properties of the classical CBO for finite $N$ and finite $\alpha$. Our approach follows a hierarchical methodology, motivated by the existing work on deterministic particle-based systems. In the linear setting, convergence to consensus has been successfully shown for multi-agent systems \cite{ode_dynamics_on_graphs}, and strict Lyapunov functions have been established for consensus dynamics under directed connected graphs \cite{ode_global_stability, multi_agent_lyapunov}. These studies leverage a graph-theoretic interpretation of particle systems, wherein agents are represented as nodes and their interactions as weighted edges.

Building upon these insights, we adopt the following structured approach:
\begin{itemize}
    \item First, in section \ref{sec:deterministic_system} we study the deterministic system
    \begin{equation}\label{DCBO_system_description}
        d X_{t}^{n} = -\lambda\left(X_{t}^{n}-\nu_{f}^{\alpha}(X_{t})\right) d t, \quad n=1,\ldots,N,
    \end{equation}
    leveraging existing results from linear graph-theoretic consensus models to establish global exponential stability of the component-wise non-linear system (\ref{DCBO_system_description}) and its explicit Euler discretization for $N < \infty$ and $\alpha < \infty$.
    \item Finally, in section \ref{sec:brownian_system} we investigate the full system governed by a SDE
    \begin{equation}\label{CBO_version_system_description}
        d Z_{t}^{n} = -\lambda\left(Z_{t}^{n}-\nu_{f}^{\alpha}(Z_{t})\right) d t + \sigma \left(Z_{t}^{n}-\nu_{f}^{\alpha}(Z_{t})\right) \circ d W_t^{n}, \quad n=1,\ldots,N,
    \end{equation}
    which incorporates both drift and diffusion effects. Here, \( \circ \) denotes the element-wise (Hadamard) product \citenote{Definition 7.5.1}{hadamard_product}. We rigorously establish both almost sure and mean square exponential stability of the system \eqref{CBO_version_system_description} and its Euler-Maruyama discretization under suitable conditions on the parameters. In our analysis, we focus on the anisotropic diffusion term $\sigma (Z_{t}^{n}-\nu_{f}^{\alpha}(Z_{t}))$ \cite{anisotropic_convergence_CBO}, rather than the more general formulation $\sigma \|Z_{t}^{n}-\nu_{f}^{\alpha}(Z_{t})\|$, where $\|\cdot\|$ typically denotes the Euclidean norm $\|\cdot\|_2$ \cite{carrillo_analytical_framework_CBO, fornasier_CBO_converge_globally}. We restrict ourselves to the diffusion term $\sigma (Z_{t}^{n}-\nu_{f}^{\alpha}(Z_{t}))$ since the graphical interpretation of the system is sufficient to reflect its stability properties.
\end{itemize}
By systematically progressing from the deterministic case to the full stochastic system, we aim to gain a deeper understanding of CBO dynamics beyond the mean-field limit, providing new theoretical insights into its behavior for practical, finite-agent settings.

\subsection{Notation}
\label{sec:notation}

To establish a consistent notation for our analysis, we introduce the following definitions. We denote \( D \) as the dimension, and the collective state of all particles at time \( t \) by the vector \( X_t = (X_t^1, \ldots, X_t^N)^\top \in \mathbb{R}^{ND} \), where each particle \( X_t^n \) is represented as \( X_t^n = (x_t^{n,1}, \ldots, x_t^{n,D})^\top \in \mathbb{R}^{D},\ n=1, \dots, N \). 

We use \( I_N \) to denote the $(N\times N)$-identity matrix. Furthermore, we define \( \mathbf{0}_N := (0, \dots, 0)^\top \in \mathbb{R}^{N} \) and \( \mathbf{1}_N := (1, \dots, 1)^\top \in \mathbb{R}^{N} \) as the all-zero and all-one column vectors of size \( N \), respectively. Unless stated otherwise, we denote by \( \|\cdot\| \) the standard Euclidean norm.

\bigskip

\section{Purely Deterministic System}
\label{sec:deterministic_system}
	
We begin by examining the purely deterministic system (\ref{DCBO_system_description}). In this setting, the exploitative diffusion is neglected, focusing solely on drift forces that drive particles towards a consensus state. Before we will establish our stability results, we precisely define what we mean by stability in this context.

Intuitively, stability describes the persistence of solutions under small perturbations: if a system is slightly disturbed, it shall return to its original state. More formally, stability is often characterized through equilibrium points. 

\begin{definition}[Equilibrium Point \citenote{\text{Definiton 3.2}}{ode_lyapunov_stability}]\label{def:equilibrium points}
    The state \( X^{\star} \in \mathbb{R}^{N} \) is said to be an \textit{equilibrium point} of an ordinary differential equation (ODE) \begin{align}\label{dyn_syst-det-gen} dX_t=a(X_t)dt, \quad t \geq 0, \end{align} if once $X_t$ is equal to $X^{\star}$, it remains equal to $X^{\star}$ for all future times, i.e., $X^{\star}$ has to satisfy $a(X^{\star})=\mathbf{0}_N$.
\end{definition}

In particular, we distinguish between stability and exponential stability.

\begin{definition}[Stability for ODE \citenote{\text{Definiton 3.3 and 3.6}}{ode_lyapunov_stability}]\label{def:ode_stability}
    The equilibrium point \( X^{\star} = \mathbf{0}_N \) is said to be \textit{stable} if, for any \( R > 0 \), there exists \( r > 0 \), such that if the initial state satisfies \( \| X_0 \| < r \), then \( \| X_t \| < R \) for all \( t \geq 0 \). Otherwise, the equilibrium point is said to be \textit{unstable}. 
    If stability holds for any initial state $X_0$, the equilibrium point is said to be \textit{globally stable}.
\end{definition}

Exponential stability strengthens the notion of stability.

\begin{definition}[Exponential Stability for ODE \citenote{\text{Definition 3.5 and 3.6}}{ode_lyapunov_stability}]
    \label{def:ode_exponential_stability}
    The equilibrium point \( X^{\star} = \mathbf{0}_N \) is \textit{exponentially stable} if there exist constants \( c > 0 \) and \( \delta > 0 \) and some \( r > 0 \) such that \( \| X_0 \| < r \) implies
    \[
    \| X_t \| \leq c \, \| X_0 \| \, e^{-\delta t}, \quad  t \geq 0.
    \]
    If the inequality holds for all \( X_0 \in \mathbb{R}^N \), the equilibrium point is called \textit{globally exponentially stable}.
\end{definition}

Exponential stability not only ensures that solutions converge to the equilibrium, but also quantifies how fast this convergence occurs. More precisely, the distance to the equilibrium decays at least as fast as \( e^{-\delta t} \) for some rate \( \delta > 0 \).

Notice, that both definitions apply specifically to the equilibrium point $X^{\star}=\mathbf{0}_N$. The stability of other equilibrium points or entire sets of points can be analyzed through a suitable transformation of the state space.

\vspace{1em}

To assess the stability properties of equilibrium points, one may employ a fundamental mathematical tool: Lyapunov functions. These functions provide a systematic approach to determining whether solutions remain close to an equilibrium or converge to it over time.

\begin{definition}[Lyapunov functions for ODE \citenote{\text{Definiton 3.8}}{ode_lyapunov_stability}]\label{def:ode_lyapunov_function}
Let  \( V: \mathbb{R}^{N} \to \mathbb{R} \) be a function,
\begin{itemize}
    \item[(1)]   which is  positive definite, i.e., we have  $V(x)=0$ for  $x=\mathbf{0}_N$ and $V(x)>0$ else, 
    \item[(2)] which has continuous partial derivatives and
    \item[(3)]whose time derivative along any state trajectory of the system $dX_t=a(X_t)dt$ is negative semi-definite, i.e.,  
    \[ \forall t \geq 0: \quad 
    \dot{V}(X_t) := \frac{dV(X_t)}{dt} = \left( \nabla V(X_t) \right)^{\top}\frac{dX_t}{dt} = \left( \nabla V(X_t) \right)^{\top} a(X_t) \leq 0.
    \]
\end{itemize}
    Then \( V \) is said to be a \textit{Lyapunov function} for the system (\ref{dyn_syst-det-gen}).
\end{definition}

The reformulation in Definition \ref{def:ode_lyapunov_function} follows directly from the chain rule. Here, we exploit the fact that $V$ depends solely on the state $X_t$ and not explicitly on time $t$.

The core idea is to construct a scalar function $V$ that decreases along system trajectories. If such a function exists, it serves as a certificate for stability.

\begin{theorem}[Stability for ODEs \citenote{\text{Theorem 3.2}}{ode_lyapunov_stability}]\label{thm:ode_local_stability} Consider the system
\eqref{dyn_syst-det-gen} and assume that \( X^{\star} = \mathbf{0}_N \) is an equilibrium point.
If there exists an Lyapunov function $V$ for this system, then 
    \( X^{\star} = \mathbf{0}_N \) is globally stable.
\end{theorem}

If $V$ has additional properties, we can obtain exponential stability.

\begin{theorem}[Exponential Stability for ODE \citenote{\text{Theorem 4.5}}{exponential_stability_ODE_2}]
    \label{thm:ode_exponential_stability}
    Consider the system
\eqref{dyn_syst-det-gen} and assume that \( X^{\star} = \mathbf{0}_N \) is an equilibrium point.
    Assume that \( V: \mathbb{R}^N \to \mathbb{R} \) is Lyapunov function (for this system), which has the following additional property:  there exist constants \( c_1, c_2, c_3, c_4 > 0 \) such that
    \begin{enumerate}
        \item[\textnormal{(i)}] \( c_1 \| x \|^2 \leq V(x) \leq c_2 \| x \|^2 \) for \( x \in \mathbb{R}^N \),
        \item[\textnormal{(ii)}] \( \dot{V}(X_t) \leq -c_3 \| X_t \|^2 \) for all \( t \geq 0\),
        \item[\textnormal{(iii)}] \( \| \nabla V(x) \| \leq c_4 \| x\| \) for \( x \in \mathbb{R}^N \).
    \end{enumerate}
    Then the equilibrium point \( X^{\star} = \mathbf{0}_N \) is globally exponentially stable with $c=\sqrt{\frac{c_2}{c_1}}$ and $\delta = \frac{c_3}{2c_2}$.
\end{theorem}
Lyapunov functions are a classical concept for deriving (exponential) stability of ODEs, and  we apply it to our dynamical system \eqref{DCBO_system_description}. During the course of our analysis it will turn out that an appropriate transformation of the state space strikingly clarifies the stability properties of the dynamics of \eqref{DCBO_system_description}.

\vspace{1em}

\subsection{System Description}
\label{sec:deterministic_system_description}

The purely deterministic system is given by
\begin{equation}\label{deterministic_basic_system_description}
    d X_{t}^{n} = -\lambda\left(X_{t}^{n}-\nu_{f}^{\alpha}(X_{t})\right) d t
\end{equation}
for $n=1, \ldots, N$. By definition of the consensus point $ \nu_{f}^{\alpha}(X_{t})$ in (\ref{eq:consensus_point}) we have 
\begin{equation*}
    d X_{t}^{n} = - \lambda \left( X_{t}^{n} - \sum_{m=1}^{N} a_m(X_t) X_{t}^{m} \right) d t 
\end{equation*}
for $n=1, \ldots, N$ with weights 
\begin{equation*}
a_{m}(X_t) = \frac{\exp (-\alpha f(X_{t}^{m}))}{\sum_{k=1}^{N} \exp (-\alpha f(X_{t}^{k}))}.
\end{equation*}
Note that $a_m=a_m(X_t)$ depends non-linearly on the particle vector $X_t$. In the following, we will often drop this dependence for notational easiness.

Since  
\begin{equation*}
\sum_{m=1}^{N} a_{m} = \sum_{m=1}^{N} \frac{\exp (-\alpha f(X_{t}^{m}))}{\sum_{k=1}^{N} \exp (-\alpha f(X_{t}^{k}))} = 1,
\end{equation*}
all weights sum up to $1$ and we can rewrite
\begin{equation*}
    d X_{t}^{n} = - \lambda \left( \sum_{\substack{m=1\\m\neq n}}^{N} a_m \cdot \left(X_{t}^{n} - X_{t}^{m} \right) \right) d t.
\end{equation*}
Considering all $N$ equations simultaneously and sorting them by spatial dimension $d$ yields
\begin{equation}\label{eq:deterministic_system_matrix_form}
    dX_t = d \begin{pmatrix}
        x_{t}^{1,1} \\ \vdots \\ x_{t}^{N,1} \\ \vdots \\ x_{t}^{1,D} \\ \vdots \\x_{t}^{N,D} 
    \end{pmatrix} 
    = - \lambda
    \begin{pmatrix}
        \sum\limits_{m=2}^{N} a_m \left(x_{t}^{1,1} - x_{t}^{m,1} \right) \\ 
        \vdots \\ 
        \sum\limits_{m=1}^{N-1} a_m \left(x_{t}^{N,1} - x_{t}^{m,1} \right) \\ 
        \vdots \\ 
        \sum\limits_{m=2}^{N} a_m \left(x_{t}^{1,D} - x_{t}^{m,D} \right) \\ 
        \vdots \\ 
        \sum\limits_{m=1}^{N-1} a_m \left(x_{t}^{N,D} - x_{t}^{m,D} \right) 
    \end{pmatrix} dt.
\end{equation}
Given the weights $a_1, \ldots, a_N$,  the evolution of each dimension is therefore independent of the others, as it depends only on the interactions between particles within the same dimension. Moreover, the weights $a_m$ are equal in all equations.

Thus, it will suffice to study the stability of the system in each dimension $d\in \{1,\ldots D\}$ on its own. If we consider the system
\begin{equation}\label{eq:deterministic_system_matrix_form_one_dimension}
    d \begin{pmatrix}
        x_{t}^{1,d} \\ \vdots \\ x_{t}^{N,d}
    \end{pmatrix} 
    = - \lambda
    \begin{pmatrix}
        \sum\limits_{m=2}^{N} a_m \left(x_{t}^{1,d} - x_{t}^{m,d} \right) \\ 
        \vdots \\ 
        \sum\limits_{m=1}^{N-1} a_m \left(x_{t}^{N,d} - x_{t}^{m,d} \right) \\ 
    \end{pmatrix} dt
\end{equation}
and introduce the notation $X_t^{:,d} = (x_{t}^{1,d},\ldots, x_{t}^{N,d})^{\top} \in \mathbb{R}^{N}$, the right-hand side of the corresponding system simplifies to
\begin{equation*}
    - \lambda \hat{L}(X_t) X_t^{:,d} dt,
\end{equation*}
where $\hat{L}(X_t)$ is defined as follows
\begin{equation}\label{eq:determinstic_L_hat}
    \hat{L}(X_t) := 
    \begin{pmatrix}
        1-a_1(X_t) & -a_2(X_t) & \cdots & -a_N(X_t) \\
        -a_1(X_t) & 1-a_2(X_t) & \cdots & -a_N(X_t) \\
        \vdots & \vdots & \ddots & \vdots \\
        -a_1(X_t) & -a_2(X_t) & \cdots & 1-a_N(X_t)
    \end{pmatrix}.
\end{equation}
The diagonal entries arise since the weights sum up to 1. This leads to the compact representation 
\begin{equation*}
    dX_t = -\lambda L(X_t)X_t dt
\end{equation*}
of system (\ref{eq:deterministic_system_matrix_form}) with 
\begin{equation}\label{eq:determinstic_L}
L(X_t) =
\begin{pmatrix}
    \hat{L}(X_t) & \mathbf{0}_{N \times N} & \cdots & \mathbf{0}_{N \times N} \\
    \mathbf{0}_{N \times N} & \hat{L}(X_t) & \cdots & \mathbf{0}_{N \times N} \\
    \vdots & \vdots & \ddots & \vdots \\
    \mathbf{0}_{N \times N} & \mathbf{0}_{N \times N} & \cdots & \hat{L}(X_t)
\end{pmatrix} \in \mathbb{R}^{DN \times DN}.
\end{equation}
Here, $\mathbf{0}_{N \times N}$ denotes the zero-matrix of dimension $N \times N$. This formulation emphasizes that each dimension $d$ can be treated separately, while still preserving the component-wise non-linear interactions between particles through the weights $a_1, \dots, a_N$.  

From a graph-theoretic perspective, the matrix $\hat{L}(X_t)$ can be seen as the Laplacian \cite{ode_global_stability} of a weighted complete graph, where the nodes represent the particles $x_{t}^{n,d}$ and the edges carry the interaction weights $-a_m(X_t)$ or $1-a_n(X_t)$. Each off-diagonal $-a_m(X_t)$ entry describes how strongly particle $n$ is influenced by particle $m$, while the diagonal terms ensure that the total influence on each particle is balanced. The motion of particle $n$ is exactly a weighted combination of the positions of all other particles, with higher weights meaning that $n$ is pulled more strongly toward those particular particles.

This framework links to the classical theory of multi-agent consensus, where stability has been fully characterized in the linear case \cite{ode_dynamics_on_graphs, ode_global_stability, multi_agent_lyapunov}. In contrast, the CBO system involves highly nonlinear weights $a_m(X_t)$, which require a refined analysis. The next section develops stability results for this nonlinear setting. In the following we will drop the dependence of $\hat{L} =\hat{L}(X_t)$ and $L=L(X_t)$ on $X_t$ for notational simplicity.

\vspace{1em}

\subsection{Global Exponential Stability}
\label{sec:deterministic_local_lyapunov_stability}

In this chapter, we establish the (global exponential) stability of the deterministic system (\ref{deterministic_basic_system_description}). As a first step, we determine the equilibrium points of said system. These correspond to solutions of the equation
\begin{equation*}
    -\lambda\left(X_{t}^{n}-\nu_{f}^{\alpha}(X_{t})\right) = \mathbf{0}_N \quad \Leftrightarrow \quad X_{t}^{n}=\nu_{f}^{\alpha}(X_{t})
\end{equation*}
for $n=1,\ldots,N$. Since the right-hand side depends on the entire particle configuration rather than on individual indices $n$, a necessary condition for an equilibrium is
\begin{equation}\label{eq:determnistic_fixed_point_necessary_conditions}
    X_{t}^{1} = \ldots = X_{t}^{N}.
\end{equation}
To verify whether this condition is also sufficient, we assume that all particles take the same value, i.e., equation \eqref{eq:determnistic_fixed_point_necessary_conditions} is satisfied. Given that $\nu_{f}^{\alpha}(X_{t})$ represents a weighted consensus point, we compute for any $n\in \{ 1,\ldots, N\}$:
\begin{equation}\label{eq:determnistic_fixed_point_sufficient_condition}
    X_{t}^{n} = \sum_{m=1}^{N} \frac{ X_{t}^{n}}{N} =\sum_{m=1}^{N} \frac{ X_{t}^{n} \exp (-\alpha f(X_{t}^{n}))}{\sum_{k=1}^{N} \exp (-\alpha f(X_{t}^{n}))} = \sum_{m=1}^{N} \frac{ X_{t}^{m} \exp(-\alpha f(X_{t}^{m}))}{\sum_{k=1}^{N} \exp (-\alpha f(X_{t}^{k}))} = \nu_{f}^{\alpha}(X_{t}).
\end{equation}
This confirms that the equilibrium points are uniquely characterized by the condition $X_{t}^{1} = \ldots = X_{t}^{N}$. Thus, all particles must reach a common consensus value at equilibrium, which belongs to the set
\begin{equation} \label{eq:deterministic_fixed_point_system} \mathcal{G}^{\star}= \left \{ x \in \mathbb{R}^{N\times D}: \, x^{1,d}= \ldots =x^{N,d}, d=1, \ldots, D\right \}. \end{equation}

\vspace{1em}

Having identified the equilibrium points, we now turn our attention to their stability properties. To this end, we focus on the matrix \( L \), which governs the dynamics of the system $dX_t = -\lambda L X_t  dt$. As previously indicated each \( d \in \{1, \ldots, D\} \) will be treated independently. We will therefore analyze the stability of the subsystem
\begin{equation}\label{eq:determinsitic_system_one_dimension}
dX_t^{:,d} = -\lambda \hat{L} X_t^{:,d}  dt
\end{equation} 
where the matrix $\hat{L}$ is given by equation \eqref{eq:determinstic_L_hat}.

Our first objective is to understand the spectral properties of \( \hat{L} \), as these determine the qualitative behavior of the system. More precisely, we will study
\begin{equation}\label{eq:determinstic_system_lyapunov_functional}
    \dot{V}(X_t^{:,d}) = ( \nabla V(X_t^{:,d}) )^{\top} a(X_t^{:,d}) = ( \nabla V(X_t^{:,d}) )^{\top} ( -\lambda \hat{L} X_t^{:,d} ) = -\lambda ( \nabla V(X_t^{:,d}) )^{\top} \hat{L} X_t^{:,d}
\end{equation}
for  a suitable function $V: \mathbb{R}^{d} \rightarrow [0, \infty)$.
For the inequality \( \dot{V}(X_t^{:,d}) \leq 0 \) to hold, we need to understand the eigenstructure of \( \hat{L} \), which reveals in which directions particles contract and in which directions motion persists or stagnates.

In order to compute the eigenvalues of \( \hat{L} \) explicitly, we make use of the following result on rank-one matrix updates, which allows for efficient determinant calculations.

\begin{lemma}[Rank-One Update \citenote{\text{p. 475}}{rank_one_update}]\label{lem:rank_one_update}
    If \( A \in \mathbb{R}^{N \times N} \) is non-singular, and if \( c \), \( d \) are \( N \)-dimensional column vectors, then:
    \begin{enumerate}
        \item[\textnormal{(i)}] \( \det (I_N + c d^{\top}) = 1 + d^{\top} c \),
        \item[\textnormal{(ii)}] \( \det (A + c d^{\top}) = \det(A) \left( 1 + d^{\top} A^{-1} c \right) \).
    \end{enumerate}
\end{lemma}

We now apply this lemma to derive the eigenvalues of \( \hat{L} \). Recall that \( \mathbf{1}_N \) denotes the all-ones column vector, and define the weight vector 
\begin{equation}
    \ell := (a_1, \ldots, a_N)^{\top} \in (0,1)^N,
\end{equation}
which collects the weights associated with each particle \( X_t^n \), \( n = 1, \ldots, N \).

\begin{lemma}[Eigenvalues of $\hat{L}$]\label{lem:eigenvalues_L_hat}
    The matrix \( \hat{L} \in \mathbb{R}^{N \times N} \) has exactly one eigenvalue equal to zero, while the remaining \( N-1 \) eigenvalues are equal to $1$, that is,
    \[
    \mu_1(\hat{L}) = 0, \quad \mu_2(\hat{L}) = \ldots = \mu_N(\hat{L}) = 1.
    \]
    A right-eigenvector corresponding to $\mu_1(\hat{L})$ is $\mathbf{1}_N$, a left-eigenvector is $\ell$.
\end{lemma}

\begin{proof}
    We begin by rewriting $\hat{L}$ using a rank-one update involving $\mathbf{1}_N$ and $\ell = (a_1,\ldots,a_N)^{\top}$:
    \begin{equation*}\label{eq:L_hat_as_rank_one_update}
        \hat{L} = 
        \begin{pmatrix}
            1-a_1 & -a_2 & \cdots & -a_N \\
            -a_1 & 1-a_2 & \cdots & -a_N \\
            \vdots & \vdots & \ddots & \vdots \\
            -a_1 & -a_2 & \cdots & 1-a_N
        \end{pmatrix} = I_N - \begin{pmatrix}
        a_1 & a_2 & \cdots & a_N \\
        a_1 & a_2 & \cdots & a_N \\
        \vdots & \vdots & \ddots & \vdots \\
        a_1 & a_2 & \cdots & a_N
        \end{pmatrix} = I_N - \mathbf{1}_N \ell^{\top}.
    \end{equation*}
    To determine the eigenvalues, we compute the characteristic equation by evaluating $\det(\hat{L} - \mu I_N) =0$. Using the rank-one update representation of $\hat{L}$ and applying (ii) of Lemma \ref{lem:rank_one_update} with $A = (1 - \mu) I_N$, $c=-\mathbf{1}_N$ and $d=\ell$, we obtain
    \begin{align*}
        \det(\hat{L} - \mu I_N) &= \det((1 - \mu) I_N) \left(1 -  \ell^{\top}((1 - \mu) I_N)^{-1}\mathbf{1}_N\right)= (1-\mu)^{N} \left(1 - \frac{1}{1 - \mu}  \ell^{\top}\mathbf{1}_N\right).
    \end{align*}
    Here, we use $A^{-1} = \frac{1}{1 - \mu} I_N$ and $\det(A) = (1-\mu)^{N}$. Since all weights sum to one, we have
    \begin{align*}
        \ell^{\top}\mathbf{1}_N = \sum_{n=1}^{N} a_n = 1.
    \end{align*}
    Substituting this into the equation gives
    \begin{align*}
        \det(\hat{L} - \mu I_N) &=  (1-\mu)^{N} \left(1 - \frac{1}{1 -\mu} \right).
    \end{align*}
    As a result, the determinant satisfies $\det(\hat{L} - \mu I_N) = 0$ if either $$(1-\mu)^{N} = 0\ \text{  or  }\ 1 - \frac{1}{1 - \mu} = 0.$$ The second conditions simplifies to $1 - \mu = 1$, which holds for $\mu=0$. The first condition implies $\mu=1$. Thus, the eigenvalues of $\hat{L}$ are $$\mu_1(\hat{L}) = 0,\quad \mu_2(\hat{L}) = \ldots = \mu_N(\hat{L}) = 1.$$
    
    For the corresponding eigenvectors to $\mu_1(\hat{L}) = 0$ we compute
    \begin{align*}
        \hat{L}\mathbf{1}_N = (I_N-\mathbf{1}_N \ell^{\top})\mathbf{1}_N = \mathbf{1}_N - \mathbf{1}_N \ell^{\top} \mathbf{1}_N = \mathbf{1}_N - \mathbf{1}_N = \mathbf{0}_N = 0 \cdot \mathbf{1}_N = \mu_1(\hat{L}) \cdot \mathbf{1}_N
    \end{align*}
    for the right-eigenvector and
    \begin{align*}
        \hat{L}^{\top}\ell = (I_N-\mathbf{1}_N \ell^{\top})^{\top} \ell = \ell - \ell \mathbf{1}_N^{\top} \ell = \ell - \ell = \mathbf{0}_N = 0 \cdot \ell = \mu_1(\hat{L}) \cdot \ell
    \end{align*}
    for the left-eigenvector. Thus, we conclude that both $\mathbf{1}_N$ and $\ell$ are right- and left-eigenvectors corresponding to the eigenvalue $\mu_1(\hat{L}) = 0$. The key property used in this derivation is that the sum of all weights equals one.
\end{proof}

Lemma \ref{lem:eigenvalues_L_hat} reveals that the matrix \( \hat{L} \) has a highly structured spectrum: it has exactly one eigenvalue equal to zero, while all remaining \( N - 1 \) eigenvalues are equal to one. The right-eigenvector associated with the zero eigenvalue is \( \mathbf{1}_N \), the all-ones vector.

By restricting condition (\ref{eq:determnistic_fixed_point_necessary_conditions}) to a fixed dimension \( d \), we find that all equilibrium points of the deterministic system $d X_t^{:,d} = - \lambda \hat{L} X_t^{:,d} dt$ satisfy
\begin{equation}\label{eq:equal_components_equilibrium}
    x_t^{1,d} = \ldots = x_t^{N,d},
\end{equation}
i.e., all particle components agree. Thus, any equilibrium point lies in the one-dimensional subspace spanned by the eigenvector \( \mathbf{1}_N \). This spectral structure admits a dynamical interpretation: the eigenspace associated with the zero eigenvalue corresponds to a direction in which the system does not evolve. Once all particles are aligned, the dynamics stops changing. In contrast, all directions orthogonal to \( \mathbf{1}_N \) correspond to the eigenvalue one and are actively contracted by the dynamics. As a result, the system is continuously pushed toward the subspace $$\mathcal{G}^{\star,d}= \{ x \in \mathbb{R}^N: \, x= \gamma \mathbf{1}_N, \, \gamma\in \mathbb{R} \}$$ of equilibrium points.

Motivated by these observations, we consider the function
\begin{equation}\label{eq:determinstic_system_lyapunov_function}
    V(X_t^{:,d}) = \left\| X_t^{:,d} - \bar{X}_t^{:,d} \mathbf{1}_N \right\|^2,
\end{equation}
where, \( \bar{X}_t^{:,d} \) denotes the average particle value in dimension \( d \), defined as
\begin{equation}\label{eq:average_particle_value_one_dimension}
    \bar{X}_t^{:,d} := \frac{1}{N} \mathbf{1}_N^{\top} X_t^{:,d} = \frac{1}{N} \sum_{n=1}^{N} x_t^{n,d}.
\end{equation}
Since $ \bar{X}_t^{:,d}$ is actually the orthogonal projection of $X_t^{:,d}$ onto $\mathcal{G}^{\star,d}$, we also have
$$ V(X_t^{:,d})= \textrm{dist}( X_t^{:,d}; \mathcal{G}^{\star,d})^2:= \inf_{y \in  \mathcal{G}^{\star,d}} \left \| X_t^{:,d} - y \right \|^2.$$
Thus, $V$ measures the squared distance of the particle system to its closest equilibrium point.  Clearly, the difference \( X_t^{:,d} - \bar{X}_t^{:,d} \mathbf{1}_N \in \mathbb{R}^{N}\) captures how far each individual particle deviates from the mean, so an alternative interpretation of $V$ is that this function  measures the squared distance of \( X_t^{:,d} \) to its mean value. However, note  that $V$ is not a classical Lyapunov function, since that we have $V(x)=0$ for all $x \in  \mathcal{G}^{\star,d}$.

Now define the auxiliary functions
$$ g :\mathbb{R}^N \rightarrow \mathbb{R}, \quad g(x)=x^{\top}x, \qquad   h :\mathbb{R}^N \rightarrow \mathbb{R}^N, \quad h(x)= x -\frac{1}{N}\mathbf{1}_N^{\top} x \mathbf{1}_N,$$
which satisfy
$$ \nabla g(x)=2x, \qquad J h(x)= I_N - \frac{1}{N} \mathbf{1}_N \mathbf{1}_N^{\top} $$
for $x \in \mathbb{R}^N$.
The expression for the Jacobi-matrix of $h$ follows from the observation
that
$$ \partial_{x_m} (h(x))_k = \partial_{x_m} \left(x_k  - \frac{1}{N}\mathbf{1}_N^{\top} x \right) = \mathbf{1}_{ \{ m=k \}} -  \frac{1}{N} $$
for $m,k=1, \ldots, N$. Here, \( \mathbf{1}_{ \{ m=k \} } \) equals $1$ when $m=k$ and $0$ otherwise.

Thus, we have
$$V(x)= g(h(x))$$
and so the chain rule yields
\begin{align*}
 \nabla V(x) & = \left( I_N - \frac{1}{N} \mathbf{1}_N \mathbf{1}_N^{\top} \right)^{\top} 2 \left( x -\frac{1}{N}\mathbf{1}_N^{\top} x  \mathbf{1}_N\right)  \\ &= 2 \left( x -\frac{1}{N}\mathbf{1}_N^{\top} x  \mathbf{1}_N \right)
 - 2  \frac{1}{N}\mathbf{1}_N \mathbf{1}_N^{\top} x + 2   \frac{1}{N^2}  \mathbf{1}_N \mathbf{1}_N^{\top} x \mathbf{1}_N^{\top} \mathbf{1}_N \\  &=
2 \left( x -\frac{1}{N}\mathbf{1}_N^{\top} x  \mathbf{1}_N \right).
 \end{align*}
Here we have used that $x^{\top} \mathbf{1}_N= \mathbf{1}_N^{\top}x$ is a scalar and $\mathbf{1}_N^T \mathbf{1}_N=N  $.

\vspace{1em}
Using this expression for $\nabla V$ in \eqref{eq:determinstic_system_lyapunov_functional}
we obtain
\begin{equation}\label{eq:deterministic_system_lyapunov_derivative_condition_applied}
    \dot{V}(X_t^{:,d}) = -2\lambda ( X_t^{:,d} - \bar{X}_t^{:,d} \mathbf{1}_N )^{\top} \hat{L} X_t^{:,d} 
\end{equation}
Now, we  will first rewrite and simplify the expression on the right-hand side of the above equality. To this end, we use the representation $\hat{L} = I_N - \mathbf{1}_N \ell^{\top}$ established in the proof of Lemma \ref{lem:eigenvalues_L_hat}. Moreover, we have $\bar{X}_t^{:,d} = \frac{1}{N} \mathbf{1}_N^{\top} X_t^{:,d}$, which yields
\begin{equation*}
    \hat{L} (\bar{X}_t^{:,d} \mathbf{1}_N) = (I_N - \mathbf{1}_N \ell^{\top}) \left(\frac{1}{N} \mathbf{1}_N^{\top} X_t^{:,d} \mathbf{1}_N\right) = \frac{1}{N} \mathbf{1}_N^{\top} X_t^{:,d} \mathbf{1}_N - \frac{1}{N} \mathbf{1}_N \ell^{\top} \mathbf{1}_N^{\top} X_t^{:,d} \mathbf{1}_N.
\end{equation*}
Since $\mathbf{1}_N^{\top} X_t^{:,d}$ is a scalar, we can rearrange terms and obtain
\begin{equation}\label{eq:nullsummand}
    \hat{L} (\bar{X}_t^{:,d} \mathbf{1}_N) = \frac{1}{N} \mathbf{1}_N^{\top} X_t^{:,d} \mathbf{1}_N - \frac{1}{N} \mathbf{1}_N^{\top} X_t^{:,d} \mathbf{1}_N \cdot \ell^{\top} \mathbf{1}_N = 0,
\end{equation}
where we used that the weights sum to one,
\begin{equation*}\label{eq:weights_sum_to_one}
    \ell^{\top} \mathbf{1}_N = \sum_{n=1}^{N} a_n = 1.
\end{equation*}

As a result, we can rewrite the total derivative as
\begin{equation*}
    \dot{V}(X_t^{:,d}) = -2\lambda ( X_t^{:,d} - \bar{X}_t^{:,d} \mathbf{1}_N )^{\top} \hat{L} X_t^{:,d} = -2\lambda ( X_t^{:,d} - \bar{X}_t^{:,d} \mathbf{1}_N )^{\top} \hat{L} ( X_t^{:,d} - \bar{X}_t^{:,d} \mathbf{1}_N ),
\end{equation*}
since the additional term vanishes due to \eqref{eq:nullsummand}. For brevity, we define the error vector $E_t^{:,d} := X_t^{:,d} - \bar{X}_t^{:,d} \mathbf{1}_N$ and compute
\begin{equation}\label{eq:lyapunov_functional_calculation}
    \begin{aligned}
        \dot{V}(X_t^{:,d}) &= -2\lambda (E_t^{:,d})^{\top} \hat{L} E_t^{:,d} \\
        &= -2\lambda (E_t^{:,d})^{\top} (I_N - \mathbf{1}_N \ell^{\top}) E_t^{:,d} \\
        &= -2\lambda (E_t^{:,d} )^{\top} E_t^{:,d} + 2\lambda (E_t^{:,d})^{\top} \mathbf{1}_N \ell^{\top} E_t^{:,d}.
    \end{aligned}
\end{equation}
To simplify the second term, we substitute back $E_t^{:,d} = X_t^{:,d} - \bar{X}_t^{:,d} \mathbf{1}_N$, use that $(X_t^{:,d})^{\top} \mathbf{1}_N$ is a scalar and $\mathbf{1}_N^{\top} \mathbf{1}_N = N$:
\begin{equation*}
    \begin{aligned}
        (E_t^{:,d})^{\top} \mathbf{1}_N \ell^{\top} 
        &= \left(X_t^{:,d} - \frac{1}{N} \mathbf{1}_N^{\top} X_t^{:,d} \mathbf{1}_N\right)^{\top} \mathbf{1}_N \ell^{\top} \\
        &= (X_t^{:,d})^{\top} \mathbf{1}_N \ell^{\top} - \frac{1}{N} \mathbf{1}_N^{\top} (X_t^{:,d})^{\top} \mathbf{1}_N \mathbf{1}_N \ell^{\top} \\
        &= (X_t^{:,d})^{\top} \mathbf{1}_N \ell^{\top} - \frac{N}{N} (X_t^{:,d})^{\top} \mathbf{1}_N \ell^{\top} \\
        &= \mathbf{0}_N^{\top}.
    \end{aligned}
\end{equation*}
Hence, the second term in \eqref{eq:lyapunov_functional_calculation} vanishes, and observing that ${V}(X_t^{:,d})=\|E_t^{:,d} \|^2$  we obtain
\begin{equation*}
    \dot{V}(X_t^{:,d}) = -2\lambda (E_t^{:,d} )^{\top} E_t^{:,d} = -2\lambda \|E_t^{:,d} \|^2 = -2\lambda V(X_t^{:,d}).
\end{equation*}
Solving this linear ODE yields
\begin{equation} \label{eq:deterministic_system_exponetial_decay} V(X_t^{:,d})= e^{-2\lambda t}  V(X_0^{:,d}), \qquad t \geq 0, \end{equation}
so this distance of the particle system to its subspace of equilibrium points decays exponentially with rate $\lambda$. 
\vspace{1em}

At this point, we require a change of coordinates, if we want to work in the setup of Definitions \ref{def:ode_stability} and \ref{def:ode_exponential_stability}, since we have a whole subspace $\mathcal{G}^{*,d}$ of equilibrium point instead of $x^{\star,d}=\mathbf{0}_N$. This is also reflected by the fact that $V$ can not be a Lyapunov function, since $V(x)=0$ for all $x \in \mathcal{G}^{\star,d}$ .

From  equation \eqref{eq:deterministic_system_exponetial_decay} we already know that
\begin{equation}\label{E-exp-stab} \|E_t^{:,d}\|^2 = e^{-2\lambda t}  \| E_0^{:,d}\|^2, \qquad t \geq 0, \end{equation}
so we will now turn to  $E_t^{:,d}  =   X_t^{:,d} - \bar{X}_t^{:,d}\mathbf{1}_N$, which is the dynamics of the particle system that is orthogonal to $\mathcal{G}^{:,d}$.  In fact, we have 
\begin{align*}
 {E}_t^{:,d} & =   X_t^{:,d} - \bar{X}_t^{:,d}\mathbf{1}_N  =  X_t^{:,d} -  \frac{1}{N} \mathbf{1}_N^T {X}_t^{:,d}  \mathbf{1}_N  =   \left( I_N - \frac{1}{N}  \mathbf{1}_N\mathbf{1}_N^T \right)  X_t^{:,d} 
\end{align*}
and so
\begin{align*}
 \dot{E}_t^{:,d} & =   -\lambda P  \hat{L} X_t^{:,d},
\end{align*}
where
$$ P= I_N - \frac{1}{N}  \mathbf{1}_N\mathbf{1}_N^T $$
is the matrix which corresponds to the projection on the orthogonal complement of $\mathcal{G}^{\star,d}$.
Since
\begin{align}\label{core-deterministic dynamics}
P \mathbf{1}_N = \mathbf{1}_N - \frac{1}{N}  \mathbf{1}_N\mathbf{1}_N^T  \mathbf{1}_N  = \mathbf{0}_N,
\end{align}
we have
\begin{align} \label{core-deterministic-2} P \hat{L}= P - P \mathbf{1}_N \ell^T= P. \end{align}
Thus, the dynamics of $ E_t^{:,d} $ is given by 
the linear ODE
$$ \dot{E}_t^{:,d} = -\lambda  {E}_t^{:,d},$$
which strikingly clarifies the dynamics of the particle system.
This ODE has the unique solution
\begin{align} \label{E-linear_ODE}  {E}_t^{:,d} = e^{-\lambda t}   {E}_0^{:,d}, \qquad t \geq 0, \end{align} which is obviously exponentially stable with rate $\lambda$ at the unique fixed point  ${E}^{\star,d}=\mathbf{0}_N$ and we trivially recover  equation \eqref{E-exp-stab}.

\vspace{1cm}

Since equation \eqref{E-linear_ODE} does not depend on the considered dimension $d$ of the particles, we have obtained the following result  for the original dynamics:

\begin{theorem}[Global Exponential Stability of the Deterministic System] \label{thm:deterministic_system_exponential_stability}
    The deterministic system $$d X_t = - \lambda L X_t dt$$ with $X_t = (x_{t}^{1,1},\ldots, x_{t}^{1,D},\ldots,x_{t}^{N,1},\ldots,x_{t}^{N,D})^{\top} \in \mathbb{R}^{DN}$ is \textbf{globally exponentially stable} with exact rate $\lambda$ on the set of equilibrium points 
   \begin{equation*}\mathcal{G}^{\star}= \left \{ x \in \mathbb{R}^{N\times D}: \, x^{1,d}= \ldots = x^{N,d}, \, d=1, \ldots, D\right \}. \end{equation*}

\end{theorem}


\subsection{Exponential Stability under Discretization}
In practical applications, usually the explicit Euler scheme
\begin{equation}\label{deterministic_basic_system_description-Euler}
     \widehat{X}_{(k+1)\Delta}^{n} =   \widehat{X}_{k\Delta}^{n} -\lambda\left( \widehat{X}_{k\Delta}^{n}-\nu_{f}^{\alpha}( \widehat{X}_{k\Delta}^{n})\right) \Delta, \qquad k=0,1, \ldots 
\end{equation}
with stepsize $\Delta>0$
is used for the simulation of 
\begin{equation*}
    d X_{t}^{n} = -\lambda\left(X_{t}^{n}-\nu_{f}^{\alpha}(X_{t})\right) d t
\end{equation*}
for $n=1, \ldots, N$. 
Proceeding as in Subsection \ref{sec:deterministic_system_description} it will again suffice to study the system in each dimension  $d\in \{1,\ldots D\}$ on its own. So, for 
$X_t^{:,d} = (x_{t}^{1,d},\ldots, x_{t}^{N,d})^{\top} \in \mathbb{R}^{N}$, the 
Euler-discretzation of
\begin{equation*}
    dX_t^{:,d}= - \lambda \hat{L}(X_t) X_t^{:,d} dt,
\end{equation*}
reads as
\begin{equation}\label{deterministic_basic_system_description-Euler-dimension}
     \widehat{X}_{(k+1)\Delta}^{:,d} =   \widehat{X}_{k\Delta}^{:,d} - \lambda \hat{L}(\widehat{X}_{k\Delta}) \widehat{X}_{k\Delta}^{:,d} \Delta, \qquad k=0,1, \ldots 
\end{equation}
where $\hat{L}(\widehat{X}_{k\Delta}) $, is defined as before, i.e., as
\begin{equation}\label{eq:determinstic_L_hat-Euler}
    \hat{L}(\widehat{X}_{k\Delta})  := 
    \begin{pmatrix}
        1-a_1(\widehat{X}_{k\Delta}) & -a_2(\widehat{X}_{k\Delta})  & \cdots & -a_N(\widehat{X}_{k\Delta})  \\
        -a_1(\widehat{X}_{k\Delta})  & 1-a_2(\widehat{X}_{k\Delta})  & \cdots & -a_N(\widehat{X}_{k\Delta})  \\
        \vdots & \vdots & \ddots & \vdots \\
        -a_1(\widehat{X}_{k\Delta})  & -a_2(\widehat{X}_{k\Delta})  & \cdots & 1-a_N(\widehat{X}_{k\Delta}) 
    \end{pmatrix}.
\end{equation}
Now we analyze also the dynamics of the Euler scheme on the orthogonal complement of $\mathcal{G}^{\star,d}$, that is of 
\begin{align*}
 \widehat{E}_{k \Delta}^{:,d} & =   \left( I_N - \frac{1}{N}  \mathbf{1}_N\mathbf{1}_N^T \right)  \widehat{X}_{k\Delta}^{:,d}.
\end{align*}
With
$$ P= I_N - \frac{1}{N}  \mathbf{1}_N\mathbf{1}_N^T  $$
we have
\begin{align*}
 \widehat{E}_{(k+1) \Delta}^{:,d} & =   P\widehat{X}_{(k+1)\Delta}^{:,d} \\
 & =  P \widehat{X}_{k\Delta}^{:,d} - \lambda P \hat{L}(\widehat{X}_{k\Delta}) \widehat{X}_{k\Delta}^{:,d} \Delta 
 \\ &=   \widehat{E}_{k \Delta}^{:,d}
 - \lambda P \widehat{X}_{k\Delta}^{:,d} \Delta + \lambda P   \mathbf{1}_N \ell(\widehat{X}_{k \Delta})^T \widehat{X}_{k \Delta}^{:,d} \Delta
 \\ &= (1-\lambda \Delta) \widehat{E}_{k \Delta}^{:,d}
\end{align*} where we have used again \eqref{core-deterministic dynamics}, i.e., $P \mathbf{1}_N= \mathbf{0}_N.$
So, we obtain
\begin{equation} \widehat{E}_{k\Delta}^{:,d} =  (1-\lambda \Delta)^{k} {E}_{0}^{:,d}, \qquad k=0,1, \ldots. \end{equation}
Since we have
$$ \ln(1-x) \leq -x, \qquad x \in [0,1),$$
we obtain 
that
 $$(1-\lambda \Delta)^{k} = \exp (k \ln (1- \lambda \Delta)) \leq \exp( -\lambda k \Delta)$$
 under the usual stability property $\Delta \lambda <1$ for the Euler scheme.
 This implies
 \begin{align}
  \|\widehat{E}_{k\Delta}^{:,d} \| \leq  e^{ -\lambda k \Delta} \|{E}_{0}^{:,d} \|,
\end{align}
so the Euler scheme preserves the globally exponential stability with rate $\lambda$ of the equilibrium points from Theorem \ref{thm:ode_exponential_stability} under $\lambda \Delta <1.$

\bigskip

\section{Stochastic System}
\label{sec:brownian_system}

We now turn to the fully stochastic formulation of CBO, which incorporates both deterministic drift and Brownian diffusion. To rigorously analyze stability in this setting, we extend the deterministic framework to stochastic differential equations of the form $d Z_t = a(Z_t) dt + b(Z_t) dW_t$. Our goal is to establish a result analogous to the global exponential stability previously shown in the deterministic setting. To this end, we introduce the necessary concepts for exponential stability.

\begin{definition}[Equilibrium Point \citenote{\text{Definiton 3.2}}{ode_lyapunov_stability}]\label{def:equilibrium_points_stoch}
    The state \( Z^{\star} \in \mathbb{R}^{N} \) is said to be an \textit{equilibrium point} of an stochastic differential equation (SDE) \begin{align}\label{dyn_syst-stoch-gen} d Z_t = a(Z_t) dt + b(Z_t) dW_t, \quad t \geq 0, \end{align} if once $Z_t$ is equal to $Z^{\star}$, it remains equal to $Z^{\star}$ for all future times, i.e., $Z^{\star}$ has to satisfy $a(Z^{\star})=\mathbf{0}_N$ and $b(Z^{\star})=\mathbf{0}_N$.
\end{definition}

\begin{definition}[Amost Sure Exponential Stability for SDE \citenote{\text{Section 4, Definition 3.1}}{sde_lyapunov_stability_mao}]\label{def:sde_exponential_stability}
    The equilibrium point \( Z^{\star} = \mathbf{0}_N \) is said to be \textit{almost surely globally exponentially stable} with rate $\delta$ if there exists a constant \( \delta > 0 \) such that for all \( Z_0 \in \mathbb{R}^{N} \) we have
    \[
    \mathbb{P} \left( \limsup_{t \to \infty} \frac{1}{t} \ln\left( \|Z_t\| \right) \leq -\delta \right) = 1.
    \]
\end{definition}

In addition to this \emph{path-wise} notion of stability, another fundamental concept is based on the behavior of the second moment. This leads to the idea of \emph{mean square stability}, where one does not analyze individual trajectories but instead considers the expected squared norm of the solution. 
	
\begin{definition}[Mean Square Exponential Stability for SDE \citenote{\text{Section 4, Definition 4.1}}{sde_lyapunov_stability_mao}]
    \label{def:sde_mean_exponential_stability}
    The equilibrium point \( Z^{\star} = \mathbf{0}_N \) is said to be \textit{mean square globally exponentially stable} with rate \( \delta \) if there exist constants \( c, \delta > 0 \) such that for all \( Z_0 \in \mathbb{R}^{N} \) we have
    \[
    \mathbb{E} \left( \|Z_t\|^2 \right) \leq c e^{-\delta t}  \mathbb{E} \left( \|Z_0\|^2 \right) \quad \text{for all } t \geq 0.
    \]
\end{definition}

In Chapter \ref{sec:deterministic_local_lyapunov_stability}, we employed Lyapunov functions to motivate a change of coordinates, specifically the orthogonal projection onto $\mathcal{G}^{\star,d}$. This allowed us to exploit a well-chosen linearization of the transformed system to establish exponential stability without further relying on Lyapunov functions. In the stochastic setting, we proceed analogously by once again applying the same change of coordinates. This enables us to derive both almost sure exponential stability and mean square exponential stability. To this end, we reformulate our system in the same spirit as in the deterministic case.



\subsection{System Description}
\label{sec:brownian_system_description}

The dynamics of the system are governed by the stochastic differential equation
\begin{equation}\label{stochastic_basic_system_description}
    d Z_{t}^{n} = -\lambda\left(Z_{t}^{n}-\nu_{f}^{\alpha}(Z_{t})\right) d t + \sigma \left(Z_{t}^{n}-\nu_{f}^{\alpha}(Z_{t})\right) \circ d W_t^{n}
\end{equation}
for $n=1, \ldots, N$, where $W_t^{1},\ldots W_t^{N}$ are independent $d$-dimensional Brownian motions, $Z_t = (Z_t^{1},\ldots,Z_t^{N})^{\top}$ represents the particle states and \( \circ \) denotes the element-wise (Hadamard) product. 

As in the deterministic case, we can reformulate $Z_{t}^{n}-\nu_{f}^{\alpha}(Z_{t})$ by expressing it in terms of pair-wise differences. This transformation can be applied to both the drift and the diffusion terms, leading to the equivalent representation
\begin{equation}\label{eq:stochastic_sum_system_description}
    d Z_{t}^{n} = -\lambda\left( \sum_{\substack{m=1\\m\neq n}}^{N} a_m \cdot (Z_{t}^{n} - Z_{t}^{m} ) \right) d t + \sigma \left( \sum_{\substack{m=1\\m\neq n}}^{N} a_m \cdot (Z_{t}^{n} - Z_{t}^{m}) \right) \circ d W_t^{n}
\end{equation}
for $n=1, \ldots, N$ with $a_m=a_m(Z_t)$ as defined in section \ref{sec:deterministic_system_description}. Additionally, we define $Z_t^{:,d} := (z_t^{1,d},\ldots,z_t^{N,d})^{\top}$  to organize the system component-wise across dimensions. Sorting the system by dimension, we obtain the following structured representation
\begin{equation}\label{eq:stochastic_complete_system_description}
    dZ_t = d \begin{pmatrix}
        z_{t}^{1,1} \\ \vdots \\ z_{t}^{N,1} \\ \vdots \\ z_{t}^{1,D} \\ \vdots \\z_{t}^{N,D} 
    \end{pmatrix} 
    = - \lambda
    \begin{pmatrix}
        \sum\limits_{m=2}^{N} a_m \left(z_{t}^{1,1} - z_{t}^{m,1} \right) \\ 
        \vdots \\ 
        \sum\limits_{m=1}^{N-1} a_m \left(z_{t}^{N,1} - z_{t}^{m,1} \right) \\ 
        \vdots \\ 
        \sum\limits_{m=2}^{N} a_m \left(z_{t}^{1,D} - z_{t}^{m,D} \right) \\ 
        \vdots \\ 
        \sum\limits_{m=1}^{N-1} a_m \left(z_{t}^{N,D} - z_{t}^{m,D} \right) 
    \end{pmatrix} dt + \sigma
    \begin{pmatrix}
        \sum\limits_{m=2}^{N} a_m \left(z_{t}^{1,1} - z_{t}^{m,1} \right) \\ 
        \vdots \\ 
        \sum\limits_{m=1}^{N-1} a_m \left(z_{t}^{N,1} - z_{t}^{m,1} \right) \\ 
        \vdots \\ 
        \sum\limits_{m=2}^{N} a_m \left(z_{t}^{1,D} - z_{t}^{m,D} \right) \\ 
        \vdots \\ 
        \sum\limits_{m=1}^{N-1} a_m \left(z_{t}^{N,D} - z_{t}^{m,D} \right) 
    \end{pmatrix} 
    \circ d \begin{pmatrix}
        w_{t}^{1,1} \\ \vdots \\ w_{t}^{N,1} \\ \vdots \\ w_{t}^{1,D} \\ \vdots \\w_{t}^{N,D} 
    \end{pmatrix}
\end{equation}
where $W_t^{n}= (w_t^{n,1},\ldots,w_t^{n,D})^{\top}$ captures the Brownian motion associated with each component.

Since all components of each Brownian motion are independent, the evolution of each dimension is independent of the others, given the weights \(a_1, \ldots, a_N\), in the same way as in the deterministic case. Consequently, it suffices once more to analyze stability separately for each dimension \(d \in \{1, \ldots, D\}\).

Using the notation from Section \ref{sec:deterministic_system}, the system dynamics for a single dimension $d$ can be expressed compactly as
\begin{equation}\label{eq:stochastic_system_description_one_dimension}
    dZ_t^{:,d} = - \lambda \hat{L} Z_t^{:,d} dt + \sigma \hat{L} Z_t^{:,d} \circ dW_t^{:,d},
\end{equation}
where $W_t^{:,d} = (w_{t}^{1,d},\ldots, w_{t}^{N,d})^{\top} \in \mathbb{R}^{N}$ represents the Brownian motion in dimension $d$. Similarly, the full system in \eqref{eq:stochastic_complete_system_description} can be rewritten in matrix form as
\begin{equation}\label{eq:stochstic_final_system_description}
    dZ_t = - \lambda L Z_t dt + \sigma L Z_t \circ dW_t
\end{equation}
with $W_t = \left(W_t^{1},\ldots, W_t^{N} \right)^{\top} \in \mathbb{R}^{DN}$. 

In the following, we extend the stability analysis to stochastic CBO dynamics. In particular, we will demonstrate that the system \eqref{eq:stochstic_final_system_description} is almost sure and mean square globally exponentially stable on the set
\begin{equation} \label{eq:stochastic_fixed_point_system} \mathcal{G}^{\star}= \left \{ z \in \mathbb{R}^{N\times D}: \, z^{1,d}= \ldots =z^{N,d}, d=1, \ldots, D\right \}, \end{equation}
which characterizes the set of equilibrium points of the system \eqref{eq:stochstic_final_system_description}. This follows directly from the same calculations as in \eqref{eq:determnistic_fixed_point_necessary_conditions} and \eqref{eq:determnistic_fixed_point_sufficient_condition}, since both the drift term $a(X_t)$ and the diffusion term $b(X_t)$ share the same functional form, differing only by a scaling factor.

\subsection{Global Exponential Stability}
\label{sec:stochastic_global_lyapunov_stability}

Since the equilibrium manifold $\mathcal{G}^\star$ remains unchanged in the stochastic setting, we employ the same change of coordinates as in the deterministic case. Furthermore, due to the independence of the Brownian motions across particles and dimensions, it suffices to analyze each dimension $d \in \{1, \ldots, D\}$ separately.

We consider the projected dynamics orthogonal to the consensus manifold by
\[
E_t^{:,d} = Z_t^{:,d} - \bar{Z}_t^{:,d} \mathbf{1}_N,
\quad \text{where} \quad
\bar{Z}_t^{:,d} = \frac{1}{N} \mathbf{1}_N^T Z_t^{:,d},
\]
so that $E_t^{:,d}$ lies in the orthogonal complement of
\[
\mathcal{G}^{\star,d} := \{ z \in \mathbb{R}^N : z = \gamma \mathbf{1}_N, \, \gamma \in \mathbb{R} \}.
\]
Analogously to the deterministic case, this leads to the convenient representation
\begin{align*}
  E_t^{:,d} = Z_t^{:,d} - \bar{Z}_t^{:,d} \mathbf{1}_N =  \left( I_N - \frac{1}{N}  \mathbf{1}_N\mathbf{1}_N^T \right)  Z_t^{:,d} = P Z_t^{:,d},
\end{align*}
where \( P \) is the projection onto the orthogonal complement of the consensus space.

In dimension \( d \), the original system evolves according to
\[
dZ_t^{:,d} = -\lambda \hat{L} Z_t^{:,d} \, dt + \sigma \hat{L} Z_t^{:,d} \circ dW_t^{:,d}.
\]
Since \( P \) is a constant matrix (independent of the particle positions), it can be pulled inside the It\^o-differential. Applying this projection yields
\[
dE_t^{:,d} = d(P Z_t^{:,d}) = -\lambda P \hat{L} Z_t^{:,d} \, dt + \sigma P \hat{L} Z_t^{:,d} \, \circ dW_t^{:,d}.
\]
Using the identity \( P \hat{L} = P \), see \eqref{core-deterministic-2}, the system simplifies to
\begin{equation} \label{eq:stoch_orth_system}
dE_t^{:,d} = -\lambda P Z_t^{:,d} \, dt + \sigma P Z_t^{:,d} \circ dW_t^{:,d} = -\lambda E_t^{:,d} \, dt + \sigma E_t^{:,d} \circ dW_t^{:,d},
\end{equation}
which takes the form of a linear stochastic differential equation.

This linear system admits the unique explicit solution
\[
E_t^{:,d} = \exp\left(-\left(\lambda + \frac{\sigma^2}{2}\right) t I + \sigma W_t^{:,d}\right) E_0^{:,d}, \quad t \geq 0,
\]
and, componentwise for each particle \( n \in \{1,\ldots, N\} \),
\begin{equation} \label{eq:one_dim_system_solution_stoch}
E_t^{n,d} = \exp\left(-\left(\lambda + \frac{\sigma^2}{2}\right)t + \sigma W_t^n \right) E_0^{n,d}, \quad t \geq 0.
\end{equation}

Applying the logarithm to the absolute value of \eqref{eq:one_dim_system_solution_stoch} and dividing by \( t > 0 \) yields
\begin{equation} \label{eq:one_dim_system_solution_stoch_log}
\frac{1}{t} \ln\left( |E_t^{n,d}| \right)
= -\left(\lambda + \frac{\sigma^2}{2}\right)
+ \sigma \frac{W_t^n}{t}
+ \frac{1}{t} \ln\left( |E_0^{n,d}| \right).
\end{equation}
By the law of the iterated logarithm \citenote{\text{Theorem 4.2}}{sde_lyapunov_stability_mao}, we have
\[
\limsup_{t \to \infty} \frac{W_t^n}{\sqrt{2t \log(\log t)}} = 1
\quad \text{almost surely}.
\]
In particular, this implies
\[
\limsup_{t \to \infty} \frac{W_t^n}{t} = 0
\quad \text{almost surely}.
\]
Moreover, since the last term in \eqref{eq:one_dim_system_solution_stoch_log} satisfies
\[
\limsup_{t \to \infty} \frac{1}{t} \ln\left( |E_0^{n,d}| \right) = 0,
\]
we may take the \(\limsup\) of \eqref{eq:one_dim_system_solution_stoch_log} to conclude
\begin{equation} \label{eq:limsup_componentwise_decay}
\limsup_{t \to \infty} \frac{1}{t} \ln\left( |E_t^{n,d}| \right)
= -\left(\lambda + \frac{\sigma^2}{2}\right)
\quad \text{almost surely}.
\end{equation}

Due to the equivalence of norms in finite-dimensional spaces, there exist constants \( c_1, c_2 > 0 \) such that
\[
c_1 \cdot \max_{n=1,\ldots,N} |E_t^{n,d}| \leq \|E_t^{:,d}\| \leq c_2 \cdot \max_{n=1,\ldots,N} |E_t^{n,d}|.
\]
Taking logarithms and dividing by \( t \) gives
\[
\frac{1}{t} \ln(c_1)
+ \frac{1}{t} \ln \left( \max_{n=1,\ldots,N} |E_t^{n,d}| \right)
\leq \frac{1}{t} \ln \left( \|E_t^{:,d}\| \right)
\leq \frac{1}{t} \ln(c_2)
+ \frac{1}{t} \ln \left( \max_{n=1,\ldots,N} |E_t^{n,d}| \right).
\]
Passing to the limit superior as \( t \to \infty \), and using \eqref{eq:limsup_componentwise_decay}, we obtain
\begin{equation}\label{eq:limsup_total_decay}
\limsup_{t \to \infty} \frac{1}{t} \ln \left( \|E_t^{:,d}\| \right)
= \limsup_{t \to \infty} \frac{1}{t} \ln \left( \max_{n=1,\ldots,N} |E_t^{n,d}| \right)
= -\left(\lambda + \frac{\sigma^2}{2}\right)
\quad \text{almost surely}.
\end{equation}

This proves that the system \eqref{eq:stoch_orth_system} is almost surely globally exponentially stable on $\mathcal{G}^{\star,d}$ with exponential rate
\[
\delta = \lambda + \frac{\sigma^2}{2}.
\]
Since this analysis holds for each dimension \( d \) independently, the result carries over to the full system, completing the proof of almost sure global exponential stability for the original dynamics.

In particular, this shows that almost sure convergence holds for all positive values of \(\lambda\) and \(\sigma\). Moreover, even if the drift parameter \(\lambda\) is negative, and the deterministic ODE would therefore be unstable, the presence of sufficiently strong noise can compensate for this and still yield almost sure stability. This phenomenon is known as \emph{stabilization by noise}, see, e.g.,  the pioneering result on page 211f in \cite{stabilization_by_noise_khasminskii2011} as well as \cite{stabilization_by_noise_Arnold, stabilization_by_noise_Mao}.

\begin{theorem}[Almost Sure Global Exponential Stability of the Stochastic System] \label{thm:stochastic_system_as_exponential_stability}
    The stochastic system $$dZ_t = - \lambda L Z_t dt + \sigma L Z_t \circ dW_t$$ with $Z_t = (z_{t}^{1,1},\ldots, z_{t}^{1,D},\ldots,z_{t}^{N,1},\ldots,z_{t}^{N,D})^{\top} \in \mathbb{R}^{DN}$ is \textbf{almost surely globally exponentially stable} with exact rate $\lambda + \frac{\sigma^2}{2}$ on the set of equilibrium points 
   \begin{equation*}\mathcal{G}^{\star}= \left \{ z \in \mathbb{R}^{N\times D}: \, z^{1,d}= \ldots = z^{N,d}, \, d=1, \ldots, D\right \}. \end{equation*}

\end{theorem}

\vspace{1.5em}

To establish not only almost sure exponential convergence but also convergence in mean square, we now consider the squared explicit solution \eqref{eq:one_dim_system_solution_stoch} of the linear SDE in expectation:
\begin{equation*}
    \mathbb{E}\left( |E_t^{n,d}|^2 \right)
    = \mathbb{E}\left( \exp\left( -2\left(\lambda + \frac{\sigma^2}{2} \right)t + 2\sigma W_t^n \right) |E_0^{n,d}|^2 \right).
\end{equation*}
Since \( W_t^n \) is independent of the initial condition \( E_0^{n,d} \), the expectation factorizes
\begin{equation*}
    \mathbb{E}\left( |E_t^{n,d}|^2 \right)
    = \exp\left( -2\left(\lambda + \frac{\sigma^2}{2} \right)t \right)
      \mathbb{E}\left( \exp(2\sigma W_t^n) \right)
      \mathbb{E}\left( |E_0^{n,d}|^2 \right).
\end{equation*}
Recalling that \( W_t^n \sim \mathcal{N}(0, t) \), the random variable \( 2\sigma W_t^n \) is Gaussian with mean \( 0 \) and variance \( 4\sigma^2 t \). For a Gaussian random variable \( X \sim \mathcal{N}(\mu, \sigma^2) \), it is well known that
\[
    \mathbb{E}\left(e^{aX}\right) = \exp\left(a\mu + \tfrac{1}{2} a^2 \sigma^2 \right).
\]
Applying this identity with \( X = W_t^n \), \( a = 2\sigma \), \( \mu = 0 \), and \( \text{Var}(W_t^n) = t \), we obtain
\[
    \mathbb{E}\left( \exp(2\sigma W_t^n) \right)
    = \exp\left( \tfrac{1}{2} (2\sigma)^2 t \right)
    = \exp\left( 2\sigma^2 t \right).
\]
Hence, the mean squared value simplifies to
\begin{align*}
\mathbb{E}\left( |E_t^{n,d}|^2 \right) &= \exp\left(-2\left(\lambda + \frac{\sigma^2}{2}\right)t \right) \exp\left( 2\sigma^2 t \right) \mathbb{E}\left(|E_0^{n,d}|^2\right) \\& = \exp\left(-2\left(\lambda - \frac{\sigma^2}{2}\right)t \right) \mathbb{E}\left(|E_0^{n,d}|^2\right).
\end{align*}
Since this relation holds for each \( n = 1, \ldots, N \), we can apply the linearity of expectation to obtain a bound for the full vector
\begin{equation}
    \mathbb{E}\left( \|E_t^{:,d}\|^2 \right) = \exp\left( -2\left( \lambda - \frac{\sigma^2}{2} \right)t \right) \mathbb{E}\left( \|E_0^{:,d}\|^2 \right).
\end{equation}
Therefore, the system \eqref{eq:stoch_orth_system} is mean square globally exponentially stable on \( \mathcal{G}^{\star,d} \), with stability constant \( c = 1 \) and exponential decay rate
\begin{equation}
    \delta = 2\lambda - \sigma^2.
\end{equation}

Since this analysis holds for each dimension \( d \) independently, the result carries over to the full system, completing the proof of mean square global exponential stability for the original dynamics.

\begin{theorem}[Mean Square Global Exponential Stability of the Stochastic System] \label{thm:stochastic_system_exponential_stability}
    The stochastic system $$dZ_t = - \lambda L Z_t dt + \sigma L Z_t \circ dW_t$$ with $Z_t = (z_{t}^{1,1},\ldots, z_{t}^{1,D},\ldots,z_{t}^{N,1},\ldots,z_{t}^{N,D})^{\top} \in \mathbb{R}^{DN}$ is \textbf{mean square globally exponentially stable} with constant $c=1$ and exact rate $2\lambda - \sigma^2$ on the set of equilibrium points 
   \begin{equation*}\mathcal{G}^{\star}= \left \{ z \in \mathbb{R}^{N\times D}: \, z^{1,d}= \ldots = z^{N,d}, \, d=1, \ldots, D\right \}. \end{equation*}

\end{theorem}

In order for the system to be mean square exponentially stable, it is necessary that the decay rate \( \delta = 2\lambda - \sigma^2 \) is positive. This requires the condition
\begin{equation}\label{eq:stability_condition}
    2\lambda > \sigma^2.
\end{equation}
We emphasize that, independently of this condition, almost sure exponential convergence always holds for all values of \( \lambda > 0 \) and \( \sigma \geq 0 \), as shown earlier. 

We note that the stability condition \eqref{eq:stability_condition} and the exponential rate \(2\lambda - \sigma^2\) obtained in Theorem \ref{thm:stochastic_system_exponential_stability} coincide with the findings of \cite{anisotropic_convergence_CBO}, where the anisotropic CBO dynamics were analyzed in the mean-field limit. In contrast, our approach neither requires structural assumptions on the objective function \(f\) nor asymptotic regimes such as \(N \to \infty\) or \(\alpha \to \infty\). Moreover, it yields the exact mean square decay rate, rather than an upper bound that can get arbitrarily close to \(2\lambda - \sigma^2\). At the same time, our analysis is restricted to the dynamics orthogonal to the consensus manifold \(\mathcal{G}^\star\), establishing exponential convergence toward \(\mathcal{G}^\star\) but not to a specific minimizer. In contrast, \cite{anisotropic_convergence_CBO} succeeded in proving convergence to the global minimizer in the mean-field setting. A precise characterization of the limiting point within \(\mathcal{G}^\star\) in our framework remains an open problem for future research.




\subsection{Exponential Stability under Discretization}

In most works on CBO, such as \cite{CBO_origin, fornasier_CBO_converge_globally, convergence_first_order_CBO, trends_in_CBO}, 
the Euler-Maruyama scheme \cite{Euler_Maruyama} is employed to simulate the dynamics of the system 
\begin{equation*}
    d Z_{t}^{n} = -\lambda\left(Z_{t}^{n}-\nu_{f}^{\alpha}(Z_{t})\right) d t + \sigma \left(Z_{t}^{n}-\nu_{f}^{\alpha}(Z_{t})\right) \circ d W_t^{n}, \qquad n=1,\ldots,N.
\end{equation*}
The Euler-Maruyama scheme then provides the discretized approximation 
\(\widehat{Z}_{k\Delta}^n\) with time step size \(\Delta > 0\), given by
\begin{equation}\label{Stochastic_basic_system_description_Euler_Maruyama}
     \widehat{Z}_{(k+1)\Delta}^{n} = \widehat{Z}_{k\Delta}^{n} -\lambda\left( \widehat{Z}_{k\Delta}^{n}-\nu_{f}^{\alpha}( \widehat{Z}_{k\Delta})\right) \Delta 
     + \sigma \left( \widehat{Z}_{k\Delta}^{n}-\nu_{f}^{\alpha}( \widehat{Z}_{k\Delta}) \right) \circ \Delta W_k^n,
\end{equation}
where \( \circ \) denotes the element-wise (Hadamard) product. Here we use the abbreviation $\Delta W_k^n = W_{(k+1)\Delta}^n - W_{k\Delta}^n$ for the increments of the Brownian motion $W^n$. Note that $\Delta W_k^n  \sim \mathcal{N}(\mathbf{0}_D, \Delta \cdot I_D)$ and moreover the random vectors $\Delta W_0^n, \ldots, \Delta W_k^n, \ldots $ are independent. Thus, the Euler–Maruyama scheme combines a forward Euler step for the drift with a pointwise scaled Gaussian perturbation to model the stochastic diffusion component.

As in Subsection \ref{sec:brownian_system_description}, it suffices to study the dynamics of each coordinate \( d \in \{1,\ldots, D\} \) separately. For $Z_t^{:,d} := (Z_t^{1,d}, \ldots, Z_t^{N,d})^\top \in \mathbb{R}^N$ the stochastic system
\[
dZ_t^{:,d} = - \lambda \hat{L}(Z_t) Z_t^{:,d} \, dt + \sigma \hat{L}(Z_t) Z_t^{:,d} \circ dW_t^{:,d}
\]
can be approximated via the Euler–Maruyama scheme as
\begin{equation}\label{eq:euler_maruyama_L_hat_per_dimension}
\widehat{Z}_{(k+1)\Delta}^{:,d} = \widehat{Z}_{k\Delta}^{:,d} - \lambda \hat{L}(\widehat{Z}_{k\Delta}) \widehat{Z}_{k\Delta}^{:,d} \Delta + \sigma \hat{L}(\widehat{Z}_{k\Delta}) \widehat{Z}_{k\Delta}^{:,d} \circ \Delta W_k^{:,d}, \qquad k = 0,1,\ldots,
\end{equation}
where \( \Delta W_k^{:,d} \in \mathbb{R}^N \) are the increments of the Brownian motions which are acting on the dimension $d$ of the particles. Here, the matrix \( \hat{L}(\widehat{Z}_{k\Delta}) \in \mathbb{R}^{N \times N} \) is given as in \eqref{eq:determinstic_L_hat-Euler}.

Analogously to the deterministic case, we now analyze the dynamics of the Euler–Maruyama scheme on the orthogonal complement of $\mathcal{G}^{\star,d}$, that is
\begin{align*}
 \widehat{E}_{k \Delta}^{:,d} & =   \left( I_N - \frac{1}{N}  \mathbf{1}_N \mathbf{1}_N^\top \right)  \widehat{Z}_{k\Delta}^{:,d} = P \widehat{Z}_{k\Delta}^{:,d}.
\end{align*}
We obtain the dynamics
\begin{align*}
 \widehat{E}_{(k+1) \Delta}^{:,d} & = P \widehat{Z}_{(k+1)\Delta}^{:,d} \\
 & = P \left( \widehat{Z}_{k\Delta}^{:,d} - \lambda \hat{L}(\widehat{Z}_{k\Delta}) \widehat{Z}_{k\Delta}^{:,d} \Delta + \sigma \hat{L}(\widehat{Z}_{k\Delta}) \widehat{Z}_{k\Delta}^{:,d} \circ \Delta W_k^{:,d} \right) \\
 & = \widehat{E}_{k\Delta}^{:,d} - \lambda P \hat{L}(\widehat{Z}_{k\Delta}) \widehat{Z}_{k\Delta}^{:,d} \Delta + \sigma P \hat{L}(\widehat{Z}_{k\Delta}) \widehat{Z}_{k\Delta}^{:,d} \circ \Delta W_k^{:,d}\\
 & = \widehat{E}_{k\Delta}^{:,d} - \lambda P \widehat{Z}_{k\Delta}^{:,d} \Delta + \sigma P \widehat{Z}_{k\Delta}^{:,d} \circ \Delta W_k^{:,d},
\end{align*}
where we have again applied the relation $P \hat{L}(\widehat{Z}_{k\Delta})=P$.

\vspace{1em}

We now analyze the mean square stability of the Euler-Maruyama discretization of the orthogonal dynamics. From the recursion
\begin{equation}\label{eq:recursion_stochastic_case}
\widehat{E}_{(k+1)\Delta}^{:,d} = \left(\mathbf{1}_N - \lambda \Delta \mathbf{1}_N + \sigma \Delta W_k^{:,d}\right) \circ \widehat{E}_{k\Delta}^{:,d},
\end{equation}
we compute the second moment. Since the Brownian increments \(\Delta W_k^{:,d} \sim \mathcal{N}(\mathbf{0}_N, \Delta \cdot I_N)\) are independent of \(\widehat{E}_{k\Delta}^{:,d}\), we obtain
\begin{align*}
\mathbb{E}\left( \| \widehat{E}_{(k+1)\Delta}^{:,d} \|^2 \right)
&= \mathbb{E}\left( \left\| \left(\mathbf{1}_N - \lambda \Delta \mathbf{1}_N + \sigma \Delta W_k^{:,d}\right) \circ \widehat{E}_{k\Delta}^{:,d} \right\|^2 \right) \\
&= \sum_{n=1}^N \mathbb{E}\left( (1 - \lambda \Delta + \sigma \Delta W_{k}^{n,d} )^2 \right)  \mathbb{E}\left( ( \widehat{E}_{k\Delta}^{n,d} )^2 \right) \\
&= \left( (1 - \lambda \Delta)^2 + \sigma^2 \Delta \right)  \mathbb{E}\left( \| \widehat{E}_{k\Delta}^{:,d} \|^2 \right),
\end{align*}
where we used that \(\mathbb{E}(\Delta W_{k}^{n.d}) = 0\) and \(\mathbb{E}((\Delta W_{k}^{n.d})^2) = \Delta\).

By iteration, we get
\begin{equation}\label{eq:mean_square_iteration}
\mathbb{E}\left( \| \widehat{E}_{k\Delta}^{:,d} \|^2 \right) 
= \left( (1 - \lambda \Delta)^2 + \sigma^2 \Delta \right)^k \mathbb{E}\left( \| E_0^{:,d} \|^2 \right).
\end{equation}
Since for \( x \in [0,1) \) it holds that \( \ln(1-x) \leq -x \), we estimate
\[
\ln \left( (1 - \lambda \Delta)^2 + \sigma^2 \Delta \right) 
= \ln \left( 1 - ( 2 \lambda \Delta - \lambda^2 \Delta^2 - \sigma^2 \Delta) \right)
\leq - (2 \lambda \Delta - \lambda^2 \Delta^2 - \sigma^2 \Delta).
\]
Under the condition \( 2\lambda - \sigma^2 - \lambda^2 \Delta \in [0,1) \), this implies the exponential decay bound
\[
\left( (1 - \lambda \Delta)^2 + \sigma^2 \Delta \right)^k \leq \exp\left( - k \Delta \left(2 \lambda - \lambda^2 \Delta - \sigma^2 \right) \right).
\]

Hence, the Euler–Maruyama approximation \(\widehat{E}_{k\Delta}^{:,d}\) preserves mean square global exponential stablity of the equilibrium manifold with rate
\begin{equation}
\delta_{\text{EM}} := 2\lambda - \sigma^2 - \lambda^2 \Delta > 0,
\end{equation}
provided the step size \(\Delta\) satisfies
\begin{equation}
    0 < \Delta < \frac{2\lambda - \sigma^2}{\lambda^2}.
\end{equation}
More precisely, under this step size condition the Euler–Maruyama scheme satisfies
\begin{equation}
    \mathbb{E}\!\left( \| \widehat{E}_{k\Delta}^{:,d} \|^2 \right) 
    \leq \exp\!\left( - k \Delta \left( 2 \lambda - \lambda^2 \Delta - \sigma^2 \right) \right) 
    \mathbb{E}\!\left( \| E_0^{:,d} \|^2 \right),
\end{equation}
which shows that the mean square distance to the equilibrium manifold decays exponentially at rate $\delta_{\text{EM}}$.

\vspace{1em}

To establish almost sure exponential convergence of the Euler-Maruyama approximation \(\widehat{E}_{k\Delta}^{:,d}\), we first consider the evolution of its components. For each \(n\), the recursion \eqref{eq:recursion_stochastic_case} reads
\begin{equation} \label{eq:componentwise_recursion}
\widehat{E}_{(k+1)\Delta}^{n,d} = \left(1 - \lambda \Delta + \sigma \Delta W_k^{n,d} \right) \widehat{E}_{k\Delta}^{n,d},
\end{equation}
which can be iteratively expanded to yield the product representation
\[
\widehat{E}_{k\Delta}^{n,d} = \widehat{E}_{0}^{n,d} \prod_{j=0}^{k-1} \left(1 - \lambda \Delta + \sigma \Delta W_j^{n,d} \right).
\]
Taking logarithms and dividing by \(k\Delta\), we obtain
\[
\frac{1}{k \Delta} \ln \left( | \widehat{E}_{k\Delta}^{n,d} | \right) = \frac{1}{k \Delta} \ln \left( | \widehat{E}_{0}^{n,d} | \right) + \frac{1}{k \Delta} \sum_{j=0}^{k-1} \ln \left( | 1 - \lambda \Delta + \sigma \Delta W_j^{n,d} | \right).
\] 
Since the increments \(\Delta W_j^{n,d} \) are i.i.d $\mathcal{N}(0, \Delta)$, the strong law of large numbers implies
\[
\lim_{k \to \infty} \frac{1}{k \Delta} \ln \left( | \widehat{E}_{k\Delta}^{n,d} | \right) = \frac{1}{\Delta} \cdot \mathbb{E} \left( \ln \left( \left| 1 - \lambda \Delta + \sigma \sqrt{\Delta} Z \right| \right) \right) \quad \text{almost surely},
\]
where \(Z \sim \mathcal{N}(0, 1)\). 

As this asymptotic rate is identical for all components and independent of \(n\), we may pass from the componentwise behavior to the full vector norm. Using norm equivalence in finite dimensions and an argument analogous to the transition from \eqref{eq:limsup_componentwise_decay} to \eqref{eq:limsup_total_decay}, we conclude
\begin{equation}
\lim_{k \to \infty} \frac{1}{k \Delta} \ln \left( \| \widehat{E}_{k\Delta}^{:,d} \| \right) = \frac{1}{\Delta} \cdot \mathbb{E} \left( \ln \left( \left| 1 - \lambda \Delta + \sigma \sqrt{\Delta} Z\right| \right) \right) \quad \text{almost surely}.
\end{equation}

Consequently, the Euler-Maruyama approximation \(\widehat{E}_{k\Delta}^{:,d}\) converges almost surely exponentially to the equilibrium manifold if and only if the expectation on the right-hand side is negative. In that case, the almost sure exponential decay rate is given by
\begin{equation}
\delta_{\mathrm{EM}} = - \frac{1}{\Delta} \cdot \mathbb{E} \left( \ln \left( \left| 1 - \lambda \Delta + \sigma \sqrt{\Delta} Z \right| \right) \right).
\end{equation}

In the proof of Theorem 4.3, \cite{almost_sure_exponention_stability_linear_SDE} establishes the bound
\begin{equation}
\mathbb{E} \left[ \ln \left( | 1 - \lambda \Delta + \sigma \sqrt{\Delta} Z | \right) \right] \leq -\left( \lambda + \frac{\sigma^2}{2} \right) \Delta + C \Delta^2,
\end{equation}
where \(C = C(\lambda, \sigma) > 0\) is a constant depending on the drift and diffusion coefficients. Inserting this into the almost sure asymptotic expression derived above, we obtain the estimate
\begin{equation}
\lim_{k \to \infty} \frac{1}{k \Delta} \ln \left( \left\| \widehat{E}_{k\Delta}^{:,d} \right\| \right) \leq -\left( \lambda + \frac{\sigma^2}{2} - C \Delta \right) \quad \text{almost surely}.
\end{equation}
Thus, provided that
\begin{equation}
    0 < \Delta < \frac{2\lambda + \sigma^2}{2C},
\end{equation}
the Euler-Maruyama scheme exhibits almost sure exponential convergence to the equilibrium manifold with rate
\begin{equation}
\delta_{\mathrm{EM}} = \lambda + \frac{\sigma^2}{2} - C \Delta.
\end{equation}

\bigskip

\section{Numerical Results}
\label{sec:numerical_results}

In this chapter, we present numerical experiments to study the convergence behavior of both the \textbf{deterministic system}  
\begin{equation}\label{eq:evolution_determinstic_system}
    d X_{t}^{n} = -\lambda\left(X_{t}^{n}-\nu_{f}^{\alpha}(X_{t})\right) d t, \quad n=1,\ldots,N,
\end{equation}  
and the \textbf{stochastic system}  
\begin{equation}\label{eq:evolution_stochastic_system}
    d Z_{t}^{n} = -\lambda\left(Z_{t}^{n}-\nu_{f}^{\alpha}(Z_{t})\right) d t + \sigma \left(Z_{t}^{n}-\nu_{f}^{\alpha}(Z_{t})\right) \circ d W_t^{n}, \quad n=1,\ldots,N,
\end{equation}  
applied to the Rastrigin function \citenote{\text{Optimization Test Problems}}{test_functions}. This classical non-convex benchmark function is defined for \(x = (x_1,\ldots,x_D) \in \mathbb{R}^D\) as  
\begin{equation}\label{eq:rastrigin}
    f_{\text{Rastrigin}}(x) = 10D + \sum_{d=1}^D \left(x_d^2 - 10\cos(2\pi x_d)\right),
\end{equation}  
with global minimizer \(\mathbf{X}^\star = \mathbf{0}_D\). Its landscape contains many regularly spaced local minima, which makes the problem difficult for optimization methods since they can easily get trapped. In the context of CBO, the Rastrigin function is therefore well suited to test the \emph{exploration ability} of particle systems and has become a standard example in numerical studies \cite{CBO_origin, fornasier_CBO_converge_globally, anisotropic_convergence_CBO}.

Our implementation builds upon the Python-based CBO framework introduced in \cite{CBO_in_python}. The complete source code, including all function definitions, numerical solvers, and visualization scripts, is available in the Git repository accompanying this paper \cite{git_Lypaunov_stability}.

Unless stated otherwise, the following settings are used throughout all experiments:
\begin{itemize}
    \item \textbf{Dimension:} \( D = 2 \), allowing intuitive visualization of particle trajectories.
    \item \textbf{Number of particles:} \( N = 100 \), initialized from the uniform distribution \( \mathcal{U}([-5, 5]^D) \).
    \item \textbf{Discretization:} We employ the Euler–Maruyama method \cite{Euler_Maruyama} with a time step \( \Delta = 0.05 \).
    \item \textbf{Time horizon:} \( T = 4.95 \) corresponding to \( 100 \) time points (epochs).
    \item \textbf{Consensus parameters:} \( \lambda = 1 \), \( \alpha = 1000 \).
\end{itemize}
In the numerical experiments, we compare the observed convergence rates with the exponential stability results established in Theorems \ref{thm:stochastic_system_as_exponential_stability} and \ref{thm:stochastic_system_exponential_stability}, which concern the continuous system. We do not refer to the discretized rates, since the additional term $\lambda^2 \Delta$ in their decay estimate is negligible for the step size $\Delta=0.05$ considered here.

\subsection{Convergence Behavior of 2D-Rastrigin}

We start by comparing the deterministic system (\ref{eq:evolution_determinstic_system}) and the stochastic system (\ref{eq:evolution_stochastic_system}) for two diffusion values,  
\[
\sigma = 1 \quad \text{and} \quad \sigma = \sqrt{2.1}.
\]  
The first satisfies the mean square stability condition (\ref{eq:stability_condition}) from Theorem~\ref{thm:stochastic_system_exponential_stability},  
\[
2\lambda > \sigma^2,
\]  
while the second violates it. We note, that the almost sure convergence holds for any configuration of $\lambda$ and $\sigma$, due to Theorem \ref{thm:deterministic_system_exponential_stability}.

We visualize the trajectories of all \(N = 100\) particles in the different system variants on the 2D-Rastrigin function. Each row in Figure \ref{fig:rastrigin_screenshot_row_all} corresponds to one system. We show four snapshots per system: the initial state at time \(t = 0\), an intermediate state after 15 time steps (\(t = 0.75\)), another after 40 steps (\(t = 2\)), and the final state after 100 steps (\(t = 4.95\)). In each plot, the current consensus point is indicated.

The systems are ordered by increasing stochasticity, starting with the purely deterministic case. This visualization allows us to qualitatively assess the influence of noise on the convergence behavior. All views are rendered from above; the global minimum \(\mathbf{X}^\star = (0, 0)\) lies at the center of each snapshot.

\begin{figure}[h]
    \centering
    \includegraphics[width=\textwidth]{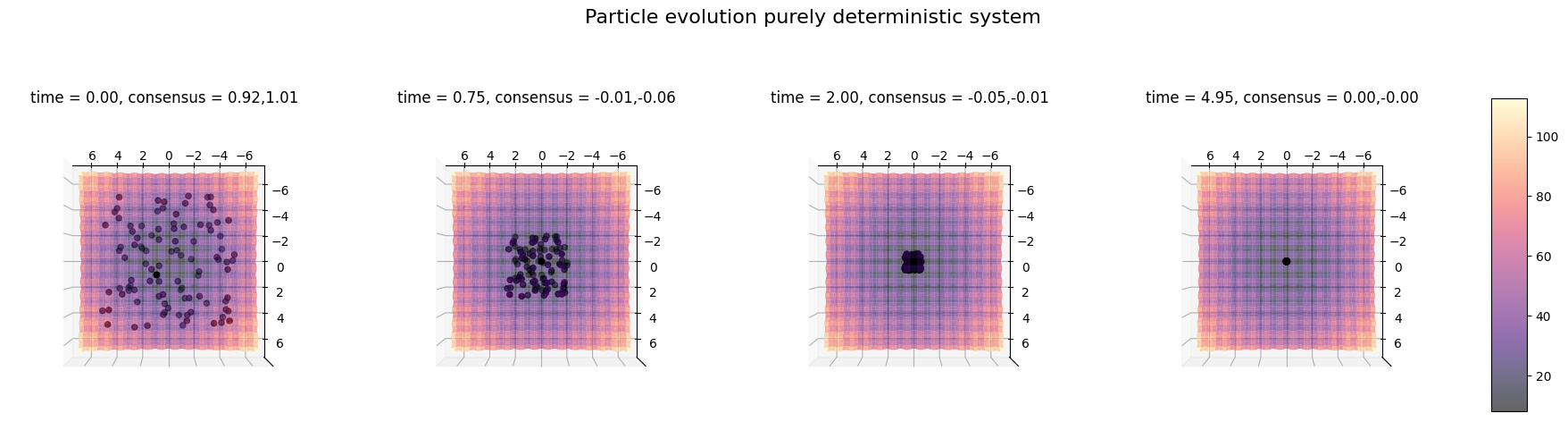}\vspace{-1ex}
    
    \includegraphics[width=\textwidth]{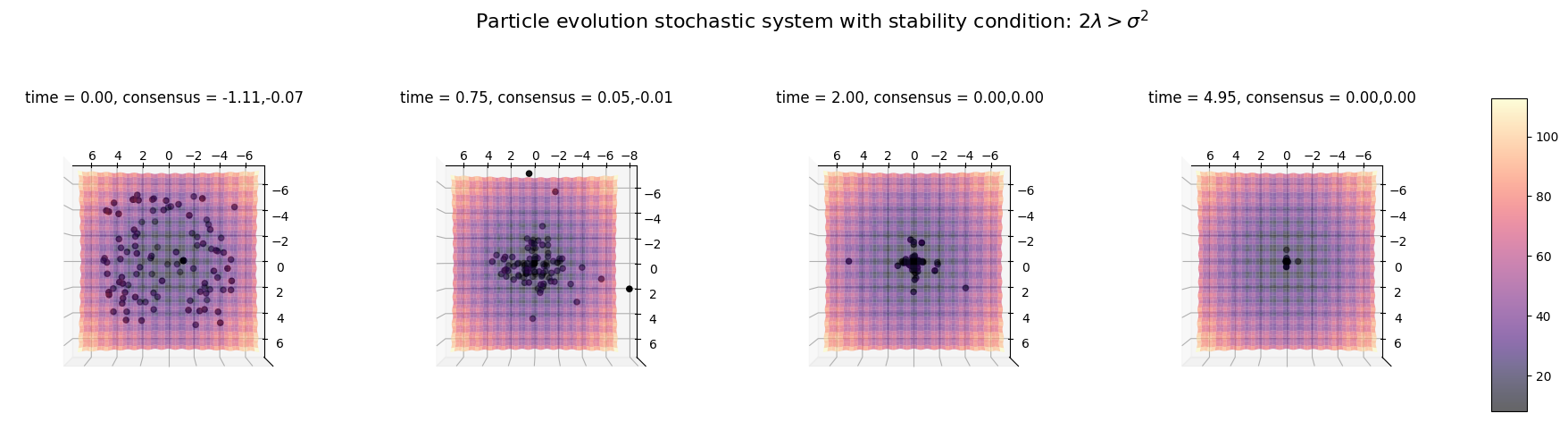}\vspace{-1ex}
    
    \includegraphics[width=\textwidth]{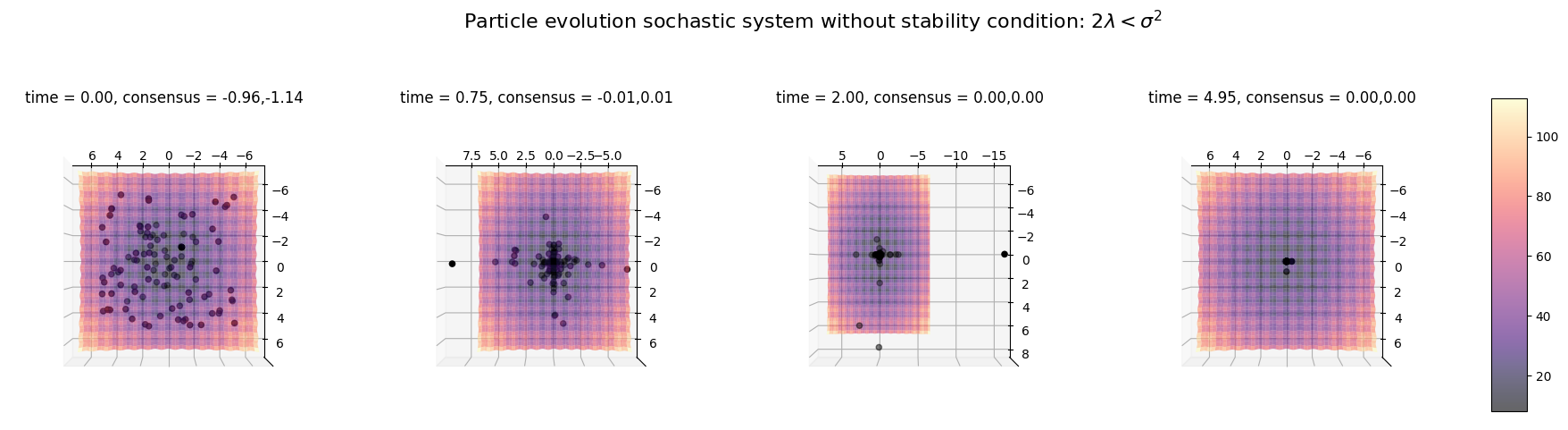}
    
    \caption{Evolution of $N = 100$ particles on the 2D-Rastrigin function for the purely deterministic system (\ref{eq:evolution_determinstic_system}),
    the stochastic system (\ref{eq:evolution_stochastic_system}) with $\sigma = 1$ (satisfying the stability condition $2 \lambda > \sigma^2$), and
    the stochastic system (\ref{eq:evolution_stochastic_system}) with $\sigma = \sqrt{2.1}$ (violating the stability condition $2 \lambda > \sigma^2$). 
    For each case, we display four snapshots over time: at initialization ($\text{time} = 0$), after 15 time steps ($\text{time} = 0.75$), after 40 time steps ($\text{time} = 2$), and at the end of the simulation ($\text{time} = 4.95$). All plots are rendered from a top-down perspective; the global minimum is located at $\mathbf{X}^{\star} = (0, 0)$.}
    \label{fig:rastrigin_screenshot_row_all} 
\end{figure}

Figure \ref{fig:rastrigin_screenshot_row_all} shows that all three systems eventually form consensus at the global minimum. The swarm of the purely deterministic system steadily contracts toward the minimum without significant deformation of the initial distribution. The roughly square-shaped initialization from the uniform distribution \( \mathcal{U}([-5, 5]^2) \) is preserved over time.

In contrast, the stochastic system exhibits more exploration, particularly when the diffusion parameter $\sigma^2 = 2.1$. The particle distribution becomes visibly more dispersed at intermediate times ($\text{time} = 0.75$ and $\text{time} = 2.0$), indicating that more noise facilitates broader exploration of the domain. Still, the system converges to consensus by the final time point, since the system converges almost surely for every combination of $\lambda$ and $\sigma$ in accordance to Theorem \ref{thm:deterministic_system_exponential_stability}.

\smallskip
\smallskip

Since all methods converge to the same minimum, we now focus on the dynamics of consensus formation, i.e., how effectively consensus emerges over time. To quantify this, we consider the expected squared distance from the consensus manifold $\mathcal{G}^{\star}$, defined as  
\[
  \mathbb{E}(\|E_t\|^2) := \sum_{d=1}^{D} \mathbb{E}\big(\|E_t^{:,d}\|^2\big)  = 
  \sum_{d=1}^{D} \mathbb{E}\big(\|P Y_t^{:,d}\|^2\big),
\]
where $P = I_N - \tfrac{1}{N}\mathbf{1}_N\mathbf{1}_N^T$ and $Y_t = Z_t$ in the stochastic case or $Y_t = X_t$ in the deterministic case. Figure~\ref{fig:rastrigin_diameter_evolution} shows the evolution of $\mathbb{E}(\|E_t\|^2)$ over time for the same three scenarios as in Figure~\ref{fig:rastrigin_screenshot_row_all}. The expectation is approximated by a Monte-Carlo simulation with 1000 runs. For comparison, the dotted lines indicate the theoretical convergence rates from Theorem~\ref{thm:stochastic_system_exponential_stability}. Note that for the stochastic system with $\sigma^2$ violating the stability condition $2\lambda > \sigma^2$, no such rate is available.

\begin{figure}[h]
    \centering
    \includegraphics[width=0.7\textwidth]{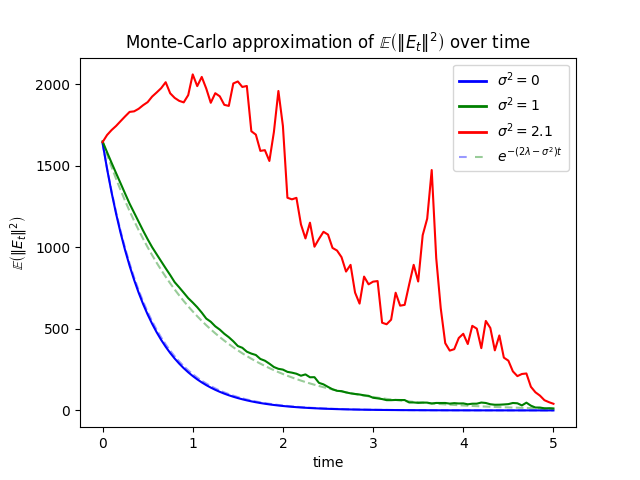}
    \caption{Monte-Carlo approximation over 1000 runs showing the evolution of the dynamics orthogonal to the consensus manifold $\mathcal{G}^{\star}$ for the purely deterministic system (\ref{eq:evolution_determinstic_system}), the stochastic system (\ref{eq:evolution_stochastic_system}) with $\sigma = 1$ (satisfying the stability condition $2 \lambda > \sigma^2$), and the stochastic system (\ref{eq:evolution_stochastic_system}) with $\sigma = \sqrt{2.1}$ (violating the stability condition). Dotted lines indicate the theoretical convergence rates from Theorem~\ref{thm:stochastic_system_exponential_stability}.}
    \label{fig:rastrigin_diameter_evolution}
\end{figure}

The purely deterministic system exhibits perfectly smooth exponential convergence, precisely matching the theoretical rate from Theorem~\ref{thm:stochastic_system_exponential_stability}. Since no randomness is present, no fluctuations occur. In contrast, the stochastic system with $\sigma^2 = 1$ shows small fluctuations, yet due to the moderate diffusion the Monte-Carlo average over 1000 runs closely follows the expected exponential rate from Theorem~\ref{thm:stochastic_system_exponential_stability}, resulting in an almost smooth convergence. For the stochastic system with $\sigma^2 = 2.1$, where the stability condition $2\lambda > \sigma^2$ is violated, randomness dominates: even after averaging over 1000 runs, pronounced fluctuations remain. At first, the strong diffusion causes the expected distance to consensus to increase, since the system explores the domain more widely. Later, however, the distance decreases again, and we still see convergence to consensus, but no longer at an exponential rate. This shows that large diffusion slows down convergence and prevents exponential stability, but consensus is still reached within the observed time horizon.

\smallskip
\smallskip

To investigate whether consensus formation behaves differently in the \textbf{isotropic} variant of CBO, we replace the anisotropic diffusion term $\sigma(Z_t^n - \nu_f^\alpha(Z_t))$ with the isotropic one $\sigma\|Z_t^n - \nu_f^\alpha(Z_t)\|_2$. We then repeat the simulations using $\sigma = \sqrt{0.5}$ and $\sigma = 1$ for the isotropic stochastic system. These values are motivated by the convergence analysis in \cite{fornasier_CBO_converge_globally}, where in the mean-field limit ($N \to \infty$, $\alpha \to \infty$) it was shown that under the isotropic stability condition
\begin{equation}\label{eq:isotropic_stability_constraint}
    2\lambda > D\sigma^2,
\end{equation}
the system converges exponentially fast with rate $2\lambda - D\sigma^2$. According to this criterion, the purely deterministic system and the case $\sigma = \sqrt{0.5}$ are exponentially stable, while $\sigma = 1$ violates the condition. Consequently, in Figure~\ref{fig:rastrigin_diameter_evolution}, the expected convergence rates (dotted lines) are shown only for the first two cases.

\begin{figure}[h]
    \centering
    \includegraphics[width=0.7\textwidth]{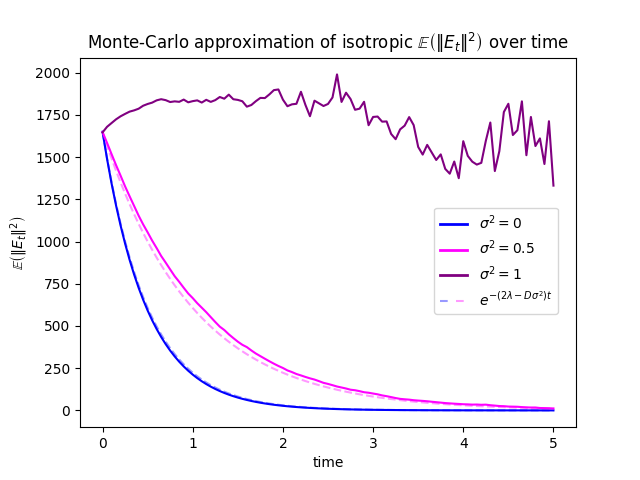}
    \caption{Monte-Carlo approximation over 1000 runs showing the evolution of the isotropic dynamics orthogonal to the consensus manifold $\mathcal{G}^{\star}$ for the purely deterministic system (\ref{eq:evolution_determinstic_system}), the isotropic stochastic system with $\sigma = \sqrt{0.5}$ (satisfying the isotropic stability condition $2 \lambda > D\sigma^2$), and the isotropic stochastic system with $\sigma = 1$ (violating the isotropic stability condition $2 \lambda > D\sigma^2$). Dotted lines indicate the theoretical convergence rates from \cite{fornasier_CBO_converge_globally}.}
    \label{fig:rastrigin_diameter_evolution}
\end{figure}

As we can see in Figure \ref{fig:rastrigin_diameter_evolution}, systems satisfying \eqref{eq:isotropic_stability_constraint} converge exponentially, matching the rates established in \cite{fornasier_CBO_converge_globally}. The system with $\sigma = 1$, which violates the stability condition, exhibits larger deviations from consensus due to stronger exploration. Within the observed time horizon, this system shows no clear convergence toward consensus.

These results indicate that the anisotropic variant generally achieves faster convergence to the minimum, even for parameters that no longer guarantee exponential stability. They also confirm that the mean-field convergence rates from \cite{fornasier_CBO_converge_globally} are approximately preserved for finite $N \ll \infty$.  

However, in the isotropic case, the diffusion term $\sigma\|Z_t^n - \nu_f^\alpha(Z_t)\|_2$ affects all dimensions jointly, depending on the positions of all particles. As a result, the approach used here for the anisotropic system cannot be directly applied, and extending the exponential stability analysis to the isotropic case is left for future work.

\subsection{Sensitivity of the convergence behavior}

We now examine how the parameters \(\alpha\), \(N\), and \(D\) affect the convergence behavior to consensus. In the following, we vary each of the three parameters individually while keeping the others constant. All simulations are averages based on 1000 independent runs. Before averaging, we limit the values of each individual run to a maximum threshold, preventing extreme outliers, particularly in large-noise systems, from dominating the observed behavior. For each configuration, we display the evolution of the dynamics orthogonal to the consensus manifold over time.

Each figure is structured into three subplots corresponding to different system. Plot (a) displays the purely deterministic system. Plot (b) uses \(\sigma = 1\), which satisfies the stability condition (\ref{eq:stability_condition}) from Theorem \ref{thm:stochastic_system_exponential_stability}. Plot (c) employs \(\sigma = \sqrt{2.1}\), which violates stability condition (\ref{eq:stability_condition}). Note that all three subplots share the same color scale to highlight differences in magnitude across the diffusion regimes. 

We start by changing the value of $\alpha$. It takes each value between 1 and 1001 with steps of 20 in between. Figure \ref{fig:rastrigin_diameter_by_alpha_MC} shows the results. 

\begin{figure}[h]
    \centering
    \includegraphics[width=1\textwidth]{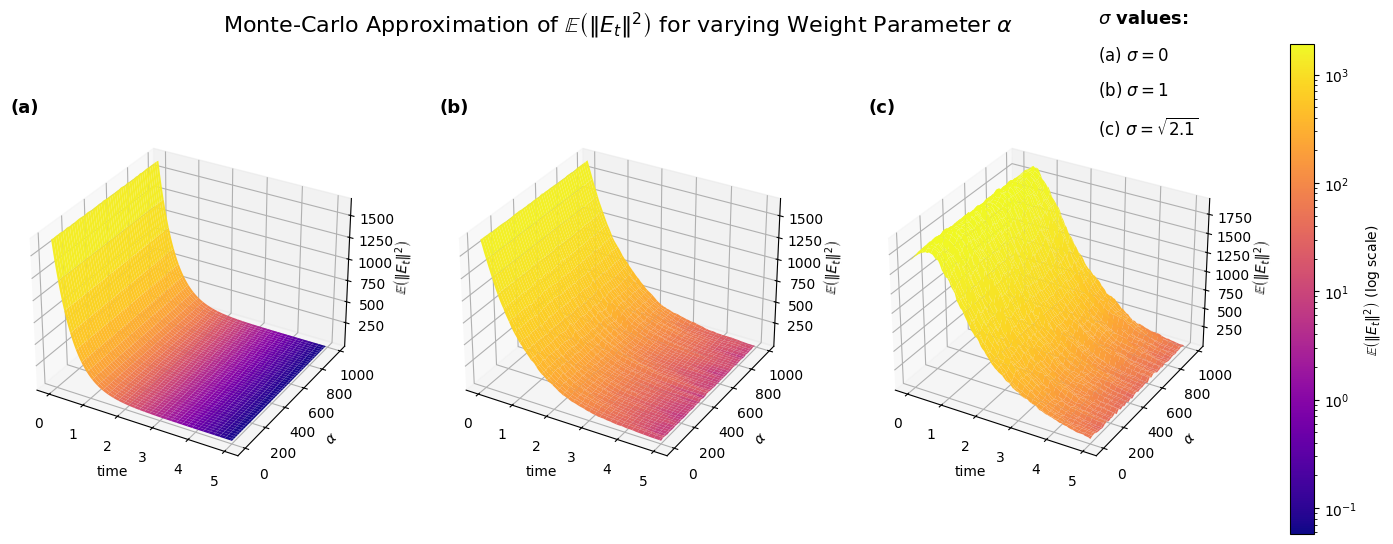}
    \caption{Monte-Carlo approximation over 1000 runs for varying $\alpha$ showing the evolution of the dynamics orthogonal to the consensus manifold $\mathcal{G}^{\star}$ for the (a) purely deterministic system (\ref{eq:evolution_determinstic_system}), the (b) stochastic system (\ref{eq:evolution_stochastic_system}) with $\sigma = 1$ (satisfying the stability condition $2 \lambda > \sigma^2$), and the (c) stochastic system (\ref{eq:evolution_stochastic_system}) with $\sigma = \sqrt{2.1}$ (violating the stability condition).}
    \label{fig:rastrigin_diameter_by_alpha_MC}
\end{figure}

In the deterministic setting (a), convergence to consensus appears smoothly and nearly identically across all values of \(\alpha\). This pattern persists in the large-diffusion cases (b) and (c), at least in expectation. The results suggest that the precise value of \(\alpha\) has only a minor influence on the overall convergence behavior.

\smallskip
\smallskip

We repeat the same procedure, this time varying only the number of particles \(N\) from 10 to 1010 in steps of 20. The results are displayed in Figure \ref{fig:rastrigin_diameter_by_n_particles_MC}.

\begin{figure}[h]
    \centering
    \includegraphics[width=1\textwidth]{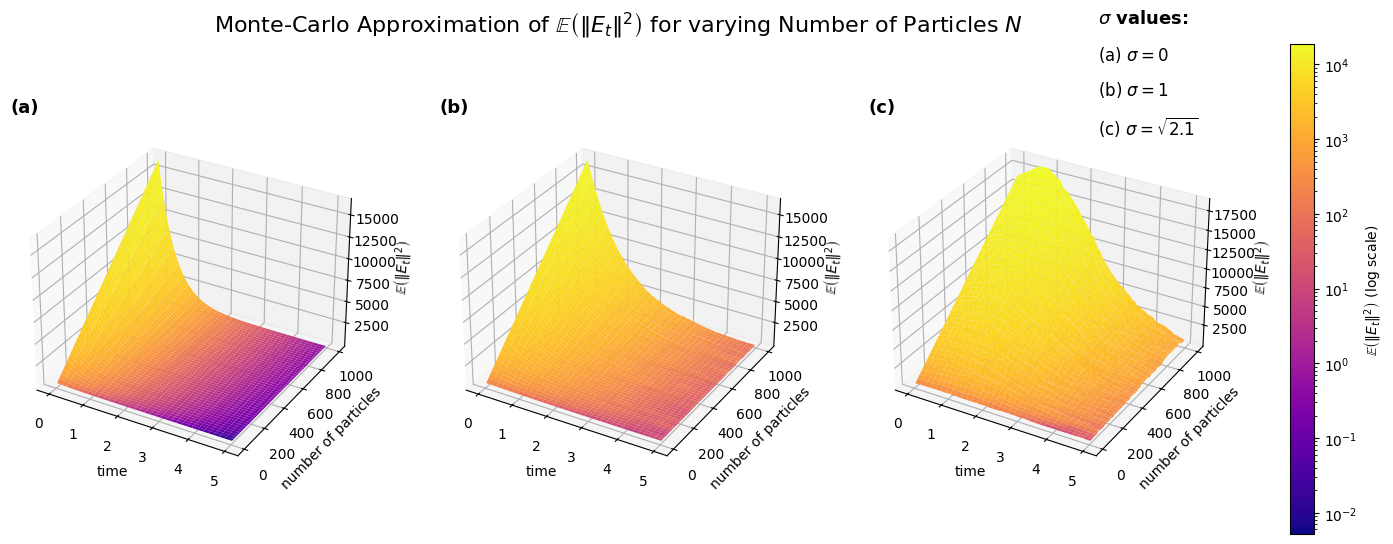}
    \caption{Monte-Carlo approximation over 1000 runs for varying number of particles $N$ showing the evolution of the dynamics orthogonal to the consensus manifold $\mathcal{G}^{\star}$ for the (a) purely deterministic system (\ref{eq:evolution_determinstic_system}), the (b) stochastic system (\ref{eq:evolution_stochastic_system}) with $\sigma = 1$ (satisfying the stability condition $2 \lambda > \sigma^2$), and the (c) stochastic system (\ref{eq:evolution_stochastic_system}) with $\sigma = \sqrt{2.1}$ (violating the stability condition).}
    \label{fig:rastrigin_diameter_by_n_particles_MC}
\end{figure}

In the purely deterministic system (a) and the system with small noise (b), increasing $N$ primarily enlarges the initial value of $\mathbb{E}(\|E_t\|^2)$, so convergence remains exponential but slows down due to the higher starting values. In contrast, in the system with large noise (c), larger $N$ not only raises the initial value of $\mathbb{E}(\|E_t\|^2)$ but also enhances exploration, causing the distance to consensus to increase temporarily before eventually decreasing. This suggests that using fewer particles can be more efficient in such scenarios.


\smallskip
\smallskip

We now turn to the question of how the system behaves as the problem dimension \(D\) increases. To investigate this, we vary \(D\) from 1 to 201 in steps of 10.

\begin{figure}[h]
    \centering
    \includegraphics[width=1\textwidth]{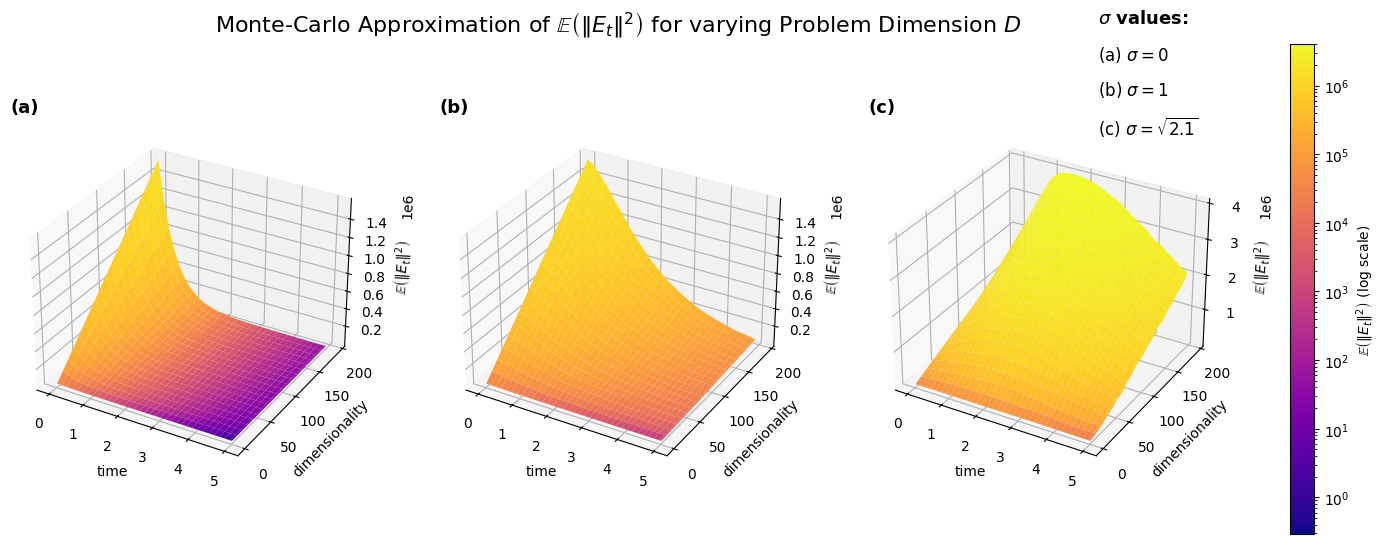}
    \caption{Monte-Carlo approximation over 1000 runs for varying $D$ showing the evolution of the dynamics orthogonal to the consensus manifold $\mathcal{G}^{\star}$ for the (a) purely deterministic system (\ref{eq:evolution_determinstic_system}), the (b) stochastic system (\ref{eq:evolution_stochastic_system}) with $\sigma = 1$ (satisfying the stability condition $2 \lambda > \sigma^2$), and the (c) stochastic system (\ref{eq:evolution_stochastic_system}) with $\sigma = \sqrt{2.1}$ (violating the stability condition).}
    \label{fig:rastrigin_diameter_by_dimensionality_MC}
\end{figure}

All three systems show that increasing \(D\) raises the initial distance to consensus, and in highly stochastic settings also enhances exploration. This illustrates how both high dimensionality and strong stochasticity can slow down consensus formation, which is consistent with prior findings \cite{CBO_adaptive_momentum, anisotropic_convergence_CBO}.

\smallskip
\smallskip

We emphasize that our analysis focuses solely on \textbf{whether consensus is achieved}, not on the \textbf{quality of the consensus point} as an optimizer, which is often investigated in other works like \cite{CBO_origin, fornasier_CBO_Sphere}. Our results show that small \(N\) and \(\sigma\) can lead to fast consensus formation, but this may result in convergence to suboptimal points. In general optimization tasks, a slow and exploratory phase can be beneficial, because the speed to consensus and solution quality do not necessarily coincide.

\subsection{Convergence Behavior of other Benchmark functions}

We also performed the same numerical simulations on two additional benchmark functions, which feature different landscape characteristics compared to the Rastrigin function.

\paragraph{Rosenbrock Function.} 
The Rosenbrock function \cite{test_functions} is a smooth test function with a narrow, curved valley and a global minimum at $\mathbf{X}^{\star} = \mathbf{1}_D$. It is given by
\begin{equation}\label{eq:rosenbrock}
    f_{\text{Rosenbrock}}(x) = \sum_{d=1}^{D-1} \left( (1 - x_d)^2 + 100(x_{d+1} - x_d^2)^2 \right),
\end{equation}
Within the CBO framework, it serves to evaluate the \textit{convergence efficiency} of the particle system.

\paragraph{Discontinuous Integrand.} 
The Discontinuous Integrand function \cite{test_functions} exhibits a steep exponential peak inside a hypercube and is zero elsewhere, creating a sharp discontinuity. This makes it suitable for testing \textit{stochastic navigation} capabilities of CBO particles in the presence of abrupt changes. It is defined as
\begin{equation}\label{eq:discontinous_integrand}
    f_{\text{disc}}(x) = 
    \begin{cases}
        - \exp\left( \sum_{d=1}^D 5 x_d \right), & \text{if } x_d \leq \frac{1}{2} \text{ for all } d = 1, \dots, D, \\
        0, & \text{otherwise}.
    \end{cases}
\end{equation}

In both cases, our simulations revealed behavior largely consistent with the results observed for the Rastrigin function, confirming the robustness of the CBO dynamics across different types of objective landscapes.

\bigskip

\section*{Acknowledgments}
This work was financially supported by the DFG Projects GO1920/11-1 and 12-1.

\bibliography{literature} 

\end{document}